\documentclass[11pt]{article}

\usepackage[T1]{fontenc}
\usepackage[margin=1in]{geometry}
\usepackage{amsmath,amssymb,amsthm}
\usepackage{mathtools}
\usepackage{microtype}
\usepackage{tikz}
\usepackage{float}
\usepackage[colorlinks=true,linkcolor=blue,citecolor=blue,urlcolor=blue]{hyperref}
\usepackage[nameinlink,noabbrev]{cleveref}

\crefname{theorem}{theorem}{theorems}
\Crefname{theorem}{Theorem}{Theorems}

\crefname{lemma}{lemma}{lemmas}
\Crefname{lemma}{Lemma}{Lemmas}

\crefname{proposition}{proposition}{propositions}
\Crefname{proposition}{Proposition}{Propositions}

\crefname{corollary}{corollary}{corollaries}
\Crefname{corollary}{Corollary}{Corollaries}

\usetikzlibrary{arrows.meta}

\usepackage{aliascnt}

\newtheorem{theorem}{Theorem}

\newaliascnt{lemma}{theorem}
\newtheorem{lemma}[lemma]{Lemma}
\aliascntresetthe{lemma}

\newaliascnt{proposition}{theorem}
\newtheorem{proposition}[proposition]{Proposition}
\aliascntresetthe{proposition}

\newaliascnt{corollary}{theorem}

\aliascntresetthe{corollary}
\numberwithin{equation}{section}

\newcommand{\R}{\mathcal R}

\title{A 64-Rectangle Counterexample to Wegner's Conjecture\\
and LP Gaps up to $5/2$}
\author{Aranya Kumar Bal\\ \small{Indian Statistical Institute, Kolkata}}
\date{}

\begin{document}

\maketitle

\begin{abstract}
Since its formulation in 1965, Wegner's conjecture has been one of the central open problems 
in geometric transversal theory, profoundly influencing research on geometric intersection graphs, 
packing and piercing, and approximation algorithms. The recent breakthrough of Ajwani, Gajjala,
Raman, and Ray settled the conjecture in the negative by constructing the first counterexample, consisting of 
$2196\cdot 8^9$ axis-parallel rectangles using a package-and-port construction alongside a computer-assisted proof, 
and also established an integrality gap of $17891/8064\approx2.21$ for the standard LP relaxation 
of the Maximum Independent Set of Rectangles problem.

In this paper, we substantially strengthen this breakthrough by giving an
explicit and conceptually simple counterexample using only $64$ rectangles.
Whereas the previous counterexample relied on a large construction with a computer-assisted
proof, ours is derived from a small geometric gadget with a recursive
composition principle alongside a hand-checkable proof, making the obstruction to Wegner's conjecture transparent.

Our construction starts from an eight-rectangle gadget whose independent sets
admit an injective assignment to four ordered labels. Combining four
horizontal and four vertical copies of this gadget yields the
$64$-rectangle counterexample. Iterating the same horizontal-vertical
construction produces a recursive family $P_r$ satisfying
$\nu(P_r)=4^{2^r}$. Already at the second level, the standard clique
relaxation has value exceeding $2\nu(P_2)$, while at the third level it
achieves an integrality gap of $73/32$, improving on the previous best value
of $17891/8064$. We further construct explicit recursive fractional solutions
whose values converge to $(5/2)\nu(P_r)$, together with piercing sets of size
exactly $(5/2)\nu(P_r)$ at every level. Consequently, both the clique-LP gap
and the packing-piercing ratio of the recursive family converge to $5/2$,
yielding rectangle families with packing-piercing ratios arbitrarily close to
$5/2$ from below. Finally, by taking disjoint unions with isolated
rectangles, we show that every rational number in $[1,5/2)$ occurs both as a
standard LP gap for Maximum Independent Set of Rectangles and as a
packing-piercing ratio of a suitable rectangle family.
\end{abstract}

\section{Introduction}\label{sec:introduction}

Let $\R$ be a finite family of closed axis-parallel rectangles in the plane.
Write $\nu(\R)$ for the maximum size of a pairwise disjoint subfamily and
$\tau(\R)$ for the minimum number of points required to intersect every
rectangle. In 1965, Wegner~\cite{Wegner_1965} conjectured that
\[
\tau(\R)\le 2\nu(\R)-1.
\]
Over the past six decades, this conjecture has been one of the central open
problems in geometric transversal theory. Beyond its intrinsic appeal, it has
profoundly influenced the study of geometric intersection graphs, rectangle
packing and piercing, combinatorial optimization, and approximation
algorithms.

A considerable body of evidence supported Wegner's conjecture and suggested
that the proposed bound was essentially best possible. Fon-Der-Flaass and
Kostochka~\cite{FonDerFlaass_1993} proved that no universal coefficient below
$5/3$ is possible, while a construction of Jel{\'\i}nek, reported by Correa,
Feuilloley, P{\'e}rez-Lantero, and
Soto~\cite{correa2015independenthitting}, gives
$\tau(\R)=2\nu(\R)-4$ for every $\nu(\R)\ge4$. Chen and
Dumitrescu~\cite{Chen_2020} further investigated the sharpness of Wegner's
inequality, and Correa et al.~\cite{correa2015independenthitting} studied the
special case of diagonal-intersecting rectangles, where the packing-piercing
gap is known to lie between $2$ and $4$. Together with steadily improving
general upper bounds, these results made the existence of a counterexample
appear increasingly unlikely.

The conjecture was finally settled by Ajwani, Gajjala, Raman, and
Ray~\cite{ajwani2026counterexamplewegnersconjectureaxisparallel} in the negative, who
constructed the first counterexample and, in the same work, established an
integrality gap of $17891/8064>2.21$ for the standard LP relaxation of the
Maximum Independent Set of Rectangles problem. Their construction is both
ingenious and highly nontrivial, combining a recursive package-and-port
framework with computational verification via linear programming and dynamic
programming. Inspired by the filter-bed construction of Asplund and
Gr\"unbaum~\cite{Asplund_1960}, it produces a family consisting of
$2196\cdot8^9$ rectangles. Besides settling a sixty-year-old conjecture, this
breakthrough opened a number of natural questions. Can the failure of
Wegner's conjecture be witnessed by a substantially smaller explicit
construction? Can one replace computational certification by a direct
combinatorial argument? More broadly, can such a construction reveal new
structural insights into rectangle packing, piercing, and LP relaxations?

In this paper, we answer these questions affirmatively. We present a simple
recursive geometric construction that yields an explicit counterexample using
only $64$ rectangles, reducing the size of the first counterexample by several
orders of magnitude. More importantly, our construction replaces
computational verification with a direct combinatorial analysis, exposing the
underlying mechanism responsible for the failure of Wegner's conjecture. This
recursive viewpoint also leads to new structural results on rectangle
packing-piercing ratios and LP relaxations that substantially strengthen the
current state of the art. The precise statements of our results are presented
in Section~\ref{sec:our_results}.

\section{Our contributions}
\label{sec:our_results}

We begin with a conceptually simple counterexample to Wegner's conjecture.
Our construction is based on an eight-rectangle gadget equipped with four
ordered labels, called \emph{slots}. These slots regulate how recursive copies
interact and provide a direct combinatorial certificate for bounding the size
of independent sets. Combining four horizontal and four vertical copies of the
gadget yields an explicit counterexample with only $64$ rectangles.

\begin{theorem}[A 64-rectangle counterexample]
\label{thm:counterexample}
There is a family $\R$ of $64$ axis-parallel rectangles such that
\[
\nu(\R)=16
\qquad\text{and}\qquad
\tau(\R)\ge 32.
\]
Consequently,
\[
\tau(\R)>2\nu(\R)-1.
\]
\end{theorem}

More importantly, the same recursive construction extends to an infinite family
of rectangle systems. Iterating the horizontal--vertical composition produces
rectangle families $P_0,P_1,P_2,\ldots$ satisfying
$\nu(P_r)=4^{2^r}$. While only the first nontrivial member $P_1$ is
triangle-free, the slot structure continues to control independent sets at
every level. This explicit recursive structure is sufficiently rich to admit
recursive fractional solutions for the standard clique relaxation of Maximum
Independent Set of Rectangles. Already the second level yields a clique-LP
value exceeding $2\nu(P_2)$, while the third level improves the previous best
finite integrality gap.

\begin{theorem}[A finite LP gap]
\label{thm:p3-gap}
For the standard point relaxation, equivalently the clique relaxation, for
Maximum Independent Set of Rectangles,
\[
\frac{\alpha^*(P_3)}{\nu(P_3)}
\ge
\frac{73}{32}.
\]
\end{theorem}

Beyond improving the best known finite gap, the recursive structure admits a
complete asymptotic analysis. We construct a recursive family of fractional
solutions whose values converge to $5/2$, together with explicit piercing sets
showing that the same limit is attained by the packing--piercing ratio.

\begin{theorem}[The limiting value]
\label{thm:endpoint}
For the family $P_r$ constructed in \Cref{sec:crossbar},
\[
\lim_{r\to\infty}\frac{\alpha^*(P_r)}{\nu(P_r)}
=
\lim_{r\to\infty}\frac{\tau(P_r)}{\nu(P_r)}
=
\frac52.
\]
\end{theorem}

Finally, the limiting theorem yields a complete interpolation result. By
taking sufficiently many disjoint copies of a suitable recursive construction
and adjoining isolated rectangles, every rational value below $5/2$ can be
realized exactly.

\begin{theorem}[Interpolation below $5/2$]
\label{thm:interpolation}
Every rational number $t\in[1,5/2)$ occurs as
$\alpha^*(\R)/\nu(\R)$ for some finite family $\R$ of
axis-parallel rectangles. The same is true for the
packing--piercing ratio $\tau(\R)/\nu(\R)$.
\end{theorem}

\section{Related work}\label{sec:related-work}

The quantities $\nu$ and $\tau$ place Wegner's conjecture in the classical
packing-piercing setting initiated by Hadwiger and Debrunner
\cite{Hadwiger_1957}. The inequality proposed by Wegner \cite{Wegner_1965} is
a particularly sharp form of such a statement for axis-parallel rectangles.
Gyárfás and Lehel \cite{gyarfas1985coveringcoloring} later asked for a
constant bound on $\tau(\R)/\nu(\R)$ for rectangle families. Wegner's conjecture
would have answered this with the asymptotic constant 2.

The lower-bound side developed more slowly, and for a long time stayed below
this threshold. Fon-Der-Flaass and Kostochka \cite{FonDerFlaass_1993} showed
that no universal coefficient below $5/3$ can hold. A construction of
Jelínek, reported by Correa, Feuilloley, Pérez-Lantero, and Soto
\cite{correa2015independenthitting}, gives $\tau(\R)=2\nu(\R)-4$ for every
$\nu(\R)\ge 4$, so the leading coefficient 2 would have been sharp if
Wegner's inequality had held. Ajwani, Gajjala, Raman, and Ray
\cite{ajwani2026counterexamplewegnersconjectureaxisparallel} were the first
to cross 2, and their larger weighted construction also gives the finite
standard-LP benchmark $17891/8064>2.21$. The recursive family in this paper
pushes the lower side further for rectangles: its packing-piercing ratios
and clique-relaxation gaps both tend to $5/2$.

The best general upper bounds for rectangles are much larger. Károlyi
\cite{Karolyi_1991} proved that a family of axis-parallel boxes in
$\mathbb R^d$ satisfies an
$O(\nu(\log\nu)^{d-1})$ piercing bound; in the plane this gives
$O(\nu\log\nu)$ for rectangles. Correa, Feuilloley, Pérez-Lantero, and Soto
\cite{correa2015independenthitting} improved the rectangle bound to
$O(\nu(\log\log\nu)^2)$, using in part small $\epsilon$-net bounds of Aronov,
Ezra, and Sharir \cite{Aronov_2010} and the approximation framework of
Chalermsook and Chuzhoy \cite{chalermsook2009maximumindependentrectangles}
for maximum independent set of rectangles. The contrast with higher dimensions is sharp: 
Tomon~\cite{Tomon_2023} proved that axis-parallel boxes in dimensions $d\ge 3$ do
not admit any constant bound in terms of $\nu$ alone. At the same time, some
structured classes in the plane do have linear behavior:
Chudnovsky, Spirkl, and Zerbib \cite{chudnovsky2017piercingaxisparallelboxes}
showed, among other results, that bounded-aspect-ratio rectangles can be
pierced by $O(\nu)$ points.

Two nearby comparisons help locate the obstruction more precisely. On the
rectangle side, Correa et al. \cite{correa2015independenthitting} analyze the
case in which all rectangles meet a common diagonal, where the
packing-piercing gap is bounded between 2 and 4. Chen and Dumitrescu
\cite{Chen_2020} studied Wegner's inequality for axis-parallel rectangles
from another direction. These results help explain why the conjecture was
plausible: the obstruction has to be global, and cannot be seen from only a
local rectangle pattern.

A second comparison comes from axis-parallel segments. Caoduro, Cslovjecsek,
Pilipczuk, and Węgrzycki \cite{Caoduro_2022} proved that every triangle-free
intersection graph of $n$ axis-parallel segments has an independence number at
least $n/4+\Omega(\sqrt n)$, and they gave a matching construction up to the
constant in the $\sqrt n$ term. Thus, the triangle-free route has a built-in
slack for segments. The 64-rectangle family here reaches the exact $n/4$
threshold, which is the point at which the same counting argument disproves
Wegner's rectangle inequality.

A different connection is algorithmic. Rectangle families are central in
approximation algorithms for Maximum Independent Set of Rectangles:
Chalermsook and Chuzhoy \cite{chalermsook2009maximumindependentrectangles}
gave an $O(\log\log n)$-approximation framework. Later recursive partitioning
methods led to constant-factor approximations, including the work of Mitchell
\cite{Mitchell_2022} and the $(2+\epsilon)$-approximation of Gálvez, Khan,
Mari, Mömke, Reddy, and Wiese \cite{galvez2021twoepsilonmisr}. The standard
point relaxation for this problem is equivalent, for rectangle graphs, to the
clique relaxation used in this paper.

\section{Organization of the paper}\label{sec:organization}

\Cref{sec:overview} gives an informal overview of the techniques used in the paper. \Cref{sec:base-gadget} derives the eight-rectangle configuration from the
slot requirements. \Cref{sec:crossbar} defines the operation that repeats
horizontal and vertical copies, and proves \Cref{thm:counterexample}.
\Cref{sec:finite-lp-gaps,sec:recursive-lower}
develop the fractional solutions, first through finite weightings on $P_2$
and $P_3$, and then recursively. \Cref{sec:upper-bound} gives the matching
piercing sets of size $(5/2)\nu(P_r)$. \Cref{sec:interpolation} proves
\Cref{thm:interpolation}.

\section{Overview of the results}
\label{sec:overview}
 
This section sketches the ideas behind the four results, and can be read before
the technical sections. We first collect some of the required notions we will use going forward here, and in the technical section.

All rectangles in this paper are closed and axis-parallel. For a finite
family $\R$, its intersection graph $G(\R)$ has one vertex for each rectangle
and one edge for each intersecting pair. A pairwise disjoint subfamily of
rectangles is exactly an independent set in $G(\R)$, so
\[ \nu(\R)=\alpha(G(\R)). \]

Axis-parallel rectangles have the Helly property: if every two rectangles in
a finite subfamily intersect, then all rectangles in that subfamily share a
common point. Hence a point that pierces some rectangles corresponds to a
clique in $G(\R)$, and piercing $\R$ is the same as covering the vertices of
$G(\R)$ by cliques.

We will often use the following observation in this informal form. If
$G(\R)$ is triangle-free, then no point of the plane lies in three rectangles.
Therefore every piercing point meets at most two rectangles, and
$\tau(\R)\ge |\R|/2$. If, in addition, $\nu(\R)\le |\R|/4$, then
\[ \tau(\R)\ge |\R|/2\ge 2\nu(\R)>2\nu(\R)-1. \]
Thus a triangle-free rectangle graph with independence number at most one
quarter of its vertices gives a counterexample to Wegner's conjecture.

We also use the standard clique relaxation. For a graph $G$, let
\[
\alpha^*(G)=\max\left\{\sum_{v\in V(G)}x_v:
x_v\ge 0,\ \sum_{v\in K}x_v\le 1\text{ for every clique }K\right\}.
\]
For rectangle graphs this is the same as the point relaxation, since cliques
are exactly the subfamilies pierced by a single point. The integral optimum is
$\alpha(G)=\nu(\R)$. Every piercing set gives a clique cover, and summing the
LP inequalities over the cliques in such a cover gives
\[ \alpha^*(G(\R))\le \tau(\R). \]

\subsection[A 64-rectangle counterexample]{A $64$-rectangle counterexample (\Cref{thm:counterexample})}
 
We start with a counting strategy, and then realize a small number of rectangles that realize the strategy. Suppose a triangle-free rectangle family has $n$
rectangles and no independent set larger than $n/4$. Then $\nu\le n/4$, while
triangle-freeness gives $\tau\ge n/2$, so
\[
  \tau\;\ge\;n/2\;\ge\;2\nu\;>\;2\nu-1,
\]
contradicting Wegner's bound. Everything therefore reduces to building a
triangle-free family that meets the threshold $\nu \le n/4$ exactly.
 
The construction starts from a base object of eight rectangles (see~\Cref{fig:base-rectangles}), drawn against
four reference vertical lines and four reference horizontal lines (these lines
become the \emph{slots} of~\Cref{sec:base-gadget}). The eight rectangles are arranged so that every independent set among them can be matched one-to-one to the four vertical reference lines. This matching guarantees that the independence number of the base object is at four. From one base object we form four horizontal copies and four vertical copies, positioned by a scaling rule so that a horizontal copy meets a vertical copy exactly when their reference lines prescribe it. This doubling produces a triangle-free configuration of $8 \cdot 8 = 64$ rectangles and the matching property lifts to an upper bound of $4\times 4 = 16$ on the independence number (see~\Cref{fig:p1-rectangles}). Since $16 = 64/4$, the counting argument fires: $\nu = 16$ while $\tau \ge 32$, so $\tau > 2\nu - 1 = 31$ and Wegner's conjecture is violated.
 
\subsection[A finite clique-relaxation gap]{A finite clique-relaxation gap (\Cref{thm:p3-gap})}

The subsequent results all depend on the following observation. The clique relaxation value sits between the packing and
piercing numbers,
\[
  \nu(\mathcal R)\;\le\;\alpha^{*}(\mathcal R)\;\le\;\tau(\mathcal R),
\]
because it relaxes the integral packing from below and is bounded above by any
clique cover, in particular by any piercing set.
 
The clique relaxation assigns each rectangle a nonnegative weight so that the
weights on any clique, equivalently, any single piercing point, sum to at
most one; its value is the total weight. On a triangle-free family this relaxation is already visible: every clique is a single rectangle or an intersecting pair, so giving every rectangle weight $1/2$ is feasible, with value $n/2 = 2\nu$. In particular the $64$-rectangle example gives ratio exactly $2$, and no more. To pass $2$, one has to go deeper into the construction; and in doing so, we will lose this triangle-freeness and thus have to be more clever.
 
Repeating the horizontal-vertical doubling yields a family
$P_0, P_1, P_2, \dots$, where $P_1$ is the counterexample above and the packing
number $\nu(P_r)$ grows in a controlled way. The point of the doubling is that
feasibility stays \emph{local}: any clique either lives inside a single copy, or
inside one horizontal copy meeting one vertical copy. So a weighting is feasible as soon as each copy is feasible on its own and, at every crossing, the combined
weight visible through that crossing stays at most one. Designing weightings that
carry high value while keeping this crossing weight small is a small, hand-checkable balancing problem. Carried out at the third level, it gives
\[
  \frac{\alpha^{*}(P_3)}{\nu(P_3)} \;\ge\; \frac{73}{32},
\]
which exceeds the earlier benchmark $17891/8064$.
 
\subsection[The limiting ratio 5/2]{The limiting ratio $5/2$ (\Cref{thm:endpoint})}
 
The finite gaps are the start of a sequence that converges, and the limit is
pinned down from both sides.
 
For the lower side, the crossing condition has a clean reusable form: if a
weighting keeps the weight visible through every vertical reference line at most
$1/2$, then it can be dropped into every horizontal and every vertical copy at the next level and the crossings stay feasible automatically, since $1/2 + 1/2 = 1$. Among these reusable weightings, a recursive scheme builds ones whose normalized value climbs toward $5/4$. Reusing a value-$5/4$ weighting at the next level then doubles the ratio, so $\alpha^{*}(P_r)/\nu(P_r) \to 5/2$ from below.
 
For the upper side, the same reference lines supply piercing sets. Using the grid
of reference-line intersections, one can pierce $P_r$ with $(5/2)\,\nu(P_r)$
points, built recursively so that every reference line carries either two or three chosen points, with the extra point on exactly half of the lines in each
direction. This gives $\tau(P_r) \le (5/2)\,\nu(P_r)$. The same piercing sets also explain why $5/4$ is the ceiling for the reusable weightings: their weight can be charged to the $(5/2)\,\nu(P_r)$ grid points, each of which absorbs at most $1/2$.
 
Since $\nu \le \alpha^{*} \le \tau$, both ratios are now squeezed between the two
bounds, and
\[
  \lim_{r\to\infty}\frac{\alpha^{*}(P_r)}{\nu(P_r)}
  \;=\;
  \lim_{r\to\infty}\frac{\tau(P_r)}{\nu(P_r)}
  \;=\;\frac{5}{2}.
\]
 
\subsection[Interpolation below 5/2]{Interpolation below $5/2$ (\Cref{thm:interpolation})}
 
The last result is a tuning argument. A single isolated rectangle has packing,
piercing, and relaxation value all equal to one, so its ratio is exactly $1$. When two families are placed far apart in the plane, all three quantities add, because disjoint packings combine, no point pierces both regions, and the relaxation splits into independent programs. These two observations let us dial any ratio downward.
 
Fix a rational target $t \in [1, 5/2)$. By the previous result some member $P_r$
has ratio above $t$; write $a = \nu(P_r)$ and $b$ for the relevant value on $P_r$. Taking $M$ separated copies of $P_r$ and adding $N$ isolated rectangles produces a family with ratio $(Mb + N)/(Ma + N)$. Solving for $N$ gives a positive rational multiple of $M$, so choosing $M$ large enough makes $N$ a positive integer and the ratio equal to $t$ exactly. The endpoint $t = 1$ is any nonempty family of pairwise disjoint rectangles. Running the argument once for the relaxation and once for the piercing number realizes every rational in $[1, 5/2)$ as each of the two ratios.

\section{The eight-rectangle base}\label{sec:base-gadget}

We now build the eight-vertex graph $F$ underlying the base drawing $P_0$.
For the moment, a slot is just one of the labels $0,1,\ldots,q-1$. In the
rectangle drawing these labels will correspond to distinguished vertical or
horizontal lines.

The numerical target comes from the first repetition. If the base object has
$m$ vertices and $q$ slots, then later we will take $q$ horizontal copies and
$q$ vertical copies, giving $2qm$ rectangles. The board used to bound
independent sets has only $q^2$ positions. To make this board one quarter of
the total size, we need $q^2=2qm/4$, or $m=2q$.

Thus we want $2q$ vertices whose independent sets can be assigned injectively
to $q$ allowed slots. There is also a geometric restriction: vertices sharing
a vertical slot must be nonadjacent, or the first horizontal-vertical
arrangement could create a triangle. In formulas, for the first list system
we require
\[ X(u)\cap X(v)\ne\emptyset \quad\Longrightarrow\quad uv\notin E(F). \]
Thus only pairs with disjoint lists are eligible to become edges.

The slots have a fixed order in the eventual drawing. Since a rectangle that
contains two distinguished vertical lines also contains every distinguished
vertical line between them, its $X$-list must be an interval of slots. The
simplest interval family of the right size consists of the singleton intervals,
the adjacent two-slot intervals, and the full interval. Its size is
$q+(q-1)+1=2q$, exactly the target above. The cases $q\le 3$ fail quickly.
For $q=1$, the two lists are
both $0$, so their vertices must be nonadjacent and cannot be injected into
one slot. For $q=2$, the lists are $0,1,01,01$; a vertex with list $0$
together with the two vertices with list $01$ gives three pairwise nonadjacent
vertices using only two slots. For $q=3$, the four lists $1,01,12,012$ all
pairwise intersect, so their vertices must be pairwise nonadjacent; again
four vertices cannot be injected into three slots.

The first useful case is therefore $q=4$, where the eight lists are
\[ 0,\ 1,\ 2,\ 3,\ 01,\ 12,\ 23,\ 0123. \]
We draw these lists in \Cref{fig:base-lists}. Intersecting intervals are
forbidden to be edges; disjoint intervals can be used as edges if they are
needed to control independent sets.

\begin{figure}[H]
\centering
\begin{tikzpicture}[
dot/.style={circle,fill=black,inner sep=1.8pt},
label/.style={font=\small},
listline/.style={line width=1.1pt,black},
slotlabel/.style={font=\small},
header/.style={font=\small}
]
\draw[gray!65] (-2.4,0.95) -- (-0.35,0.05);
\node[header] at (-0.95,0.72) {slots};
\node[header] at (-1.78,0.25) {vertices};
\foreach \x/\lab in {0.55/0,2.05/1,3.55/2,5.05/3}
  \node[slotlabel] at (\x,0.35) {$\lab$};

\foreach \name/\y/\a/\b in {
  {$C_0$}/-0.45/0.55/0.55,
  {$C_1$}/-1.2/2.05/2.05,
  {$C_2$}/-1.95/3.55/3.55,
  {$C_3$}/-2.7/5.05/5.05,
  {$U_L$}/-3.7/0.55/2.05,
  {$Q$}/-4.45/2.05/3.55,
  {$U_R$}/-5.2/3.55/5.05,
  {$T$}/-6.2/0.55/5.05
} {
  \node[label,anchor=east] at (-0.45,\y) {\name};
  \draw[listline] (\a,\y) -- (\b,\y);
  \node[dot] at (\a,\y) {};
  \node[dot] at (\b,\y) {};
}
\end{tikzpicture}
\caption{The first list system. Intersecting intervals cannot be joined by an
edge; disjoint intervals can be joined if needed to control independent sets.}
\label{fig:base-lists}
\end{figure}
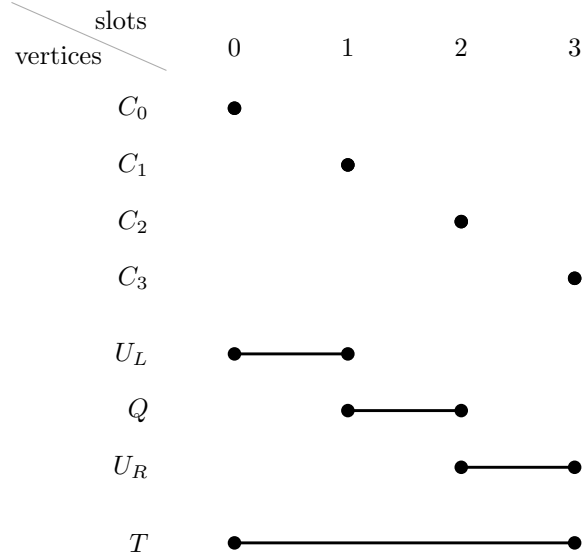

We write $[q]=\{0,1,\ldots,q-1\}$. If $F$ is a graph and each vertex $v$ has
a nonempty list $X(v)\subseteq [q]$, we say that $(F,X)$ has the Hall property
if every independent set $I$ of $F$ admits an injective map
$\phi:I\to [q]$ with $\phi(v)\in X(v)$ for every $v\in I$. For
$A\subseteq [q]$, let $F_X[A]$ be the subgraph induced by the vertices whose
lists are contained in $A$.

\begin{lemma}\label{lem:hall-criterion}
The pair $(F,X)$ has the Hall property if and only if
$\alpha(F_X[A])\le |A|$ for every $A\subseteq [q]$.
\end{lemma}

\begin{proof}
Assume first that $\alpha(F_X[A])\le |A|$ for every $A\subseteq[q]$. Let $I$
be an independent set, and let $J\subseteq I$. Put
$S_J=\bigcup_{v\in J}X(v)$. Then $J$ is an independent set in $F_X[S_J]$, so
$|J|\le |S_J|$. This is exactly Hall's condition for choosing distinct
representatives from the lists $X(v)$, $v\in I$.

Conversely, assume that $(F,X)$ has the Hall property. Any independent set in
$F_X[A]$ has all its lists contained in $A$, so its chosen representatives lie
in $A$. Therefore it has size at most $|A|$.
\end{proof}

We now choose the graph. Name the vertices according to their first list
system:
\[
\begin{array}{c|cccccccc}
v&T&U_L&U_R&C_0&C_1&C_2&C_3&Q\\ \hline
X(v)&0123&01&23&0&1&2&3&12.
\end{array}
\]
As explained above, edges may be added only between vertices with disjoint
lists.

The Hall condition tells us where edges are needed. Looking only at slots
$0,1$, the vertices $C_0,C_1,U_L$ are present. Since $U_L$ shares a slot with
both $C_0$ and $C_1$, it cannot be adjacent to either. Therefore we add the
edge $C_0C_1$; otherwise these three vertices would form an independent set
using only two slots. The same argument for slots $1,2$ and $2,3$ forces
$C_1C_2$ and $C_2C_3$.

Next look at slots $0,1,2$. The vertices $C_0,C_2,U_L,Q$ would otherwise be
independent, so we need one more edge among pairs with disjoint lists. We add
$C_0Q$. Similarly, looking at slots $1,2,3$, we add $C_3Q$. These two choices
close the five-cycle $C_0-C_1-C_2-C_3-Q-C_0$, while still respecting the rule
that edges use disjoint lists.

Finally, $T$ uses every slot, so it is isolated. The five-cycle has
independence number two. Thus, without an edge between $U_L$ and $U_R$, we
could take $T$, both vertices $U_L,U_R$, and an independent pair from the
cycle. That would give an independent set of size five.

A single edge from one of $U_L,U_R$ to the cycle does not remove this
obstruction. Such an edge rules out only those independent pairs using one
specified cycle vertex, while the five-cycle still has an independent pair
avoiding that vertex. Therefore, if $U_L$ and $U_R$ are left nonadjacent, one
would need further edges to meet all independent pairs of the cycle. Since the
lists $01$ and $23$ are disjoint, the single edge $U_LU_R$ removes the
obstruction at once.

The resulting graph is
\[ F=C_5\sqcup K_2\sqcup K_1, \]
where the five-cycle is $C_0-C_1-C_2-C_3-Q-C_0$, the edge is $U_LU_R$, and
the isolated vertex is $T$. The largest independent sets in these three
components have sizes $2,1,1$, adding exactly to the four slots available. The
graph is shown in \Cref{fig:base-graph}.

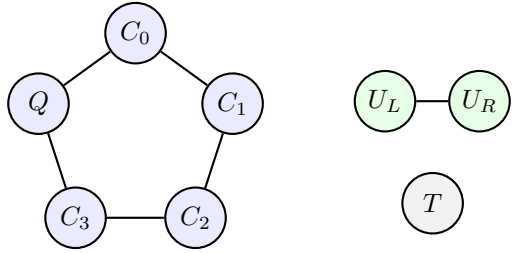
\begin{figure}[H]
\centering
\begin{tikzpicture}[
vertex/.style={circle,draw,thick,fill=white,minimum size=8mm,inner sep=1pt,font=\small},
cyclev/.style={vertex,fill=blue!8},
pairv/.style={vertex,fill=green!10},
singlev/.style={vertex,fill=gray!10},
edge/.style={thick}
]
\node[cyclev] (C0) at (90:1.35) {$C_0$};
\node[cyclev] (C1) at (18:1.35) {$C_1$};
\node[cyclev] (C2) at (-54:1.35) {$C_2$};
\node[cyclev] (C3) at (-126:1.35) {$C_3$};
\node[cyclev] (Q) at (162:1.35) {$Q$};
\draw[edge] (C0)--(C1)--(C2)--(C3)--(Q)--(C0);

\node[pairv] (UL) at (3.3,0.45) {$U_L$};
\node[pairv] (UR) at (4.55,0.45) {$U_R$};
\draw[edge] (UL)--(UR);

\node[singlev] (T) at (3.93,-0.9) {$T$};
\end{tikzpicture}
\caption{The graph $F=C_5\sqcup K_2\sqcup K_1$.}
\label{fig:base-graph}
\end{figure}

\begin{lemma}\label{lem:base-x}
The pair $(F,X)$ has the Hall property, and vertices with a common $X$-slot
are nonadjacent.
\end{lemma}

\begin{proof}
The second statement holds by construction: every edge was added only between
vertices with disjoint $X$-lists.

We verify the Hall property using \Cref{lem:hall-criterion}. If $|A|\le 1$,
then $F_X[A]$ contains at most one vertex. If $|A|=2$, the possible induced
subgraphs are
\[
\begin{array}{c|c}
A&F_X[A]\\ \hline
01&C_0C_1\sqcup U_L\\
12&C_1C_2\sqcup Q\\
23&C_2C_3\sqcup U_R\\
02&C_0\sqcup C_2\\
03&C_0\sqcup C_3\\
13&C_1\sqcup C_3.
\end{array}
\]
Each has independence number at most two. If $|A|=3$, the vertex $T$ is not
present; the available part of the five-cycle has independence number at most
two, and at most one of $U_L,U_R$ is present. Thus the independence number is
at most three. Finally,
$\alpha(F)=\alpha(C_5)+\alpha(K_2)+\alpha(K_1)=2+1+1=4$.
\end{proof}

The first list system is enough to control the crossings in the
64-rectangle family, but it does not record all the information needed for the
next round. When a copy is made after swapping the two coordinates, the old
horizontal guide lines become vertical guide lines for that copy. Thus the
base drawing should also record which horizontal guide lines pass through each
rectangle. We encode this information by a second list system $Y$.

The requirements on $Y$ are gentler than the requirements on $X$. For $X$,
sharing a slot had to force nonadjacency, because $X$ controls the mixed
intersections in the first horizontal-vertical arrangement. For $Y$, we only
ask that independent sets can choose distinct horizontal slots. Adjacent
vertices may therefore share a horizontal slot, because they never have to be
assigned slots at the same time.

This turns the choice of $Y$ into a small assignment problem: how should four
horizontal slots be spent on the three components of $F$? A largest independent
set uses at most one vertex from the isolated component, at most one endpoint
of $U_LU_R$, and at most two vertices from the five-cycle. This suggests one
slot for $T$, one shared slot for $U_L,U_R$, and two slots for the five-cycle.
Put $T$ on slot $3$, and put both $U_L$ and $U_R$ on slot $2$.

It remains to place the five-cycle on slots $0$ and $1$. If every cycle vertex
were forced to a single one of these two slots, then some nonadjacent pair
would be forced to share a slot. We instead leave one cycle vertex flexible.
For the present labeling, put the adjacent pair $C_0,Q$ on slot $0$ and the
adjacent pair $C_1,C_2$ on slot $1$. The remaining vertex $C_3$ is adjacent to
one vertex in each pair, namely $Q$ and $C_2$. If $C_3$ is selected and both
slots $0$ and $1$ are already occupied, then the other selected cycle vertices
must be $C_0$ and $C_1$, but those two are adjacent. Thus giving $C_3$ the
choice $01$ lets it avoid whichever slot is already used.

Thus the second list system is
\[
\begin{array}{c|cccccccc}
v&T&U_L&U_R&C_0&C_1&C_2&C_3&Q\\ \hline
Y(v)&3&2&2&0&1&1&01&0.
\end{array}
\]

\begin{lemma}\label{lem:base-y}
The pair $(F,Y)$ has the Hall property.
\end{lemma}

\begin{proof}
Let $I$ be an independent set. If $T\in I$, assign it to slot $3$. Since
$U_LU_R$ is an edge, at most one of $U_L,U_R$ lies in $I$; assign it to
slot $2$ if present. The remaining vertices of $I$ lie on the five-cycle, so
there are at most two of them. The vertices forced to slot $0$, namely
$C_0,Q$, are adjacent, and the vertices forced to slot $1$, namely $C_1,C_2$,
are adjacent. The only vertex with a choice is $C_3$, whose list is $01$.
If $C_3$ is present and another selected cycle vertex has already used one of
slots $0,1$, assign $C_3$ to the other slot; if $C_3$ is the only selected
cycle vertex, assign it to either slot. This gives an injective assignment for
every independent set.
\end{proof}

\begin{proposition}\label{prop:base-realization}
The graph $F$ has a rectangle realization with four distinguished vertical
lines and four distinguished horizontal lines such that the lines passing
through each rectangle are exactly the lists $X$ and $Y$.
\end{proposition}

\begin{proof}
Work in the square $[0,30]^2$. Use the following eight rectangles, drawn in
\Cref{fig:base-rectangles}:

\begin{figure}[H]
\centering
\begin{tikzpicture}[scale=0.21, every node/.style={font=\scriptsize}]
\filldraw[fill=purple!20,draw=purple!70!black,thick,fill opacity=.55] (0,26) rectangle (30,28);
\node at (15,27) {$T$};

\filldraw[fill=green!25,draw=green!45!black,thick,fill opacity=.55] (0,20) rectangle (16,22);
\node at (8,21) {$U_L$};
\filldraw[fill=green!25,draw=green!45!black,thick,fill opacity=.55] (14,20) rectangle (30,22);
\node at (22,21) {$U_R$};

\filldraw[fill=blue!18,draw=blue!65!black,thick,fill opacity=.5] (0,2) rectangle (8,10);
\node at (4,6) {$C_0$};
\filldraw[fill=blue!18,draw=blue!65!black,thick,fill opacity=.5] (6,8) rectangle (16,14);
\node at (11,11) {$C_1$};
\filldraw[fill=blue!18,draw=blue!65!black,thick,fill opacity=.5] (14,8) rectangle (24,14);
\node at (19,11) {$C_2$};
\filldraw[fill=blue!18,draw=blue!65!black,thick,fill opacity=.5] (22,2) rectangle (30,12);
\node at (26,7) {$C_3$};
\filldraw[fill=orange!25,draw=orange!70!black,thick,fill opacity=.55] (6,0) rectangle (24,4);
\node at (15,2) {$Q$};

\draw[black,thick] (0,0) rectangle (30,30);

\foreach \x/\lab in {3/0,11/1,19/2,27/3} {
  \draw[dashed,gray!70] (\x,0) -- (\x,30);
  \node[below] at (\x,0) {$\lab$};
}
\foreach \y/\lab in {3/0,11/1,21/2,27/3} {
  \draw[dashed,gray!70] (0,\y) -- (30,\y);
  \node[left] at (0,\y) {$\lab$};
}
\end{tikzpicture}
\caption{The eight rectangles realizing $F$. The dashed vertical lines give
the $X$-lists, and the dashed horizontal lines give the $Y$-lists.}
\label{fig:base-rectangles}
\end{figure}
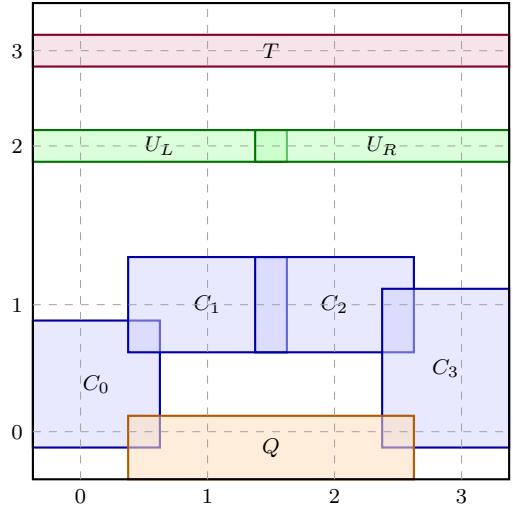

\[
\begin{array}{c|c|c}
 &x\text{-interval}&y\text{-interval}\\ \hline
T   &[0,30] &[26,28]\\
U_L &[0,16] &[20,22]\\
U_R &[14,30]&[20,22]\\
C_0 &[0,8]  &[2,10]\\
C_1 &[6,16] &[8,14]\\
C_2 &[14,24]&[8,14]\\
C_3 &[22,30]&[2,12]\\
Q   &[6,24] &[0,4].
\end{array}
\]
Let the distinguished vertical lines be $\xi=(3,11,19,27)$ and the
distinguished horizontal lines be $\eta=(3,11,21,27)$. Reading off containment
in the displayed intervals gives exactly the lists $X$ and $Y$.
We shall also use the following small separation fact: every distinguished
line is either at least one unit inside the relevant interval, or at least one
unit outside it. This is immediate from the table.

Checking interval overlap in both coordinates, the intersecting pairs are
$U_LU_R$ and
\[ C_0C_1,\ C_1C_2,\ C_2C_3,\ C_3Q,\ QC_0. \]
These are precisely the edges of $F$.
\end{proof}

\section{Repeating horizontal and vertical copies}\label{sec:crossbar}

The eight-rectangle drawing has two kinds of information. The rectangles
themselves give the graph $F$, while the distinguished lines give the two list
systems $X$ and $Y$. We now turn the preceding design into coordinates: one
drawing with such data produces a larger drawing with the same kind of data.

Suppose a rectangle family $P$ is drawn inside $[0,L]^2$ with $q$
distinguished vertical lines
$\xi_0<\xi_1<\cdots<\xi_{q-1}$ and $q$ distinguished horizontal lines
$\eta_0<\eta_1<\cdots<\eta_{q-1}$. The list $X(R)$ records the vertical
lines contained in $R$, and $Y(R)$ records the horizontal lines contained in
$R$. We assume both list systems have the Hall property. We also assume the
same separation used in \Cref{prop:base-realization}: each marker line is at
least one unit inside or at least one unit outside every relevant interval,
and consecutive marker lines are at least one unit apart in each direction.

Set $M=L+1$. For a rectangle $R=[x_1,x_2]\times[y_1,y_2]$ in $P$, define a
horizontal copy in row $i$ by
\[ H_i(R)=[Mx_1,Mx_2]\times[M\xi_i+y_1,M\xi_i+y_2], \]
and define a vertical copy in column $j$ by
\[ V_j(R)=[M\xi_j+y_1,M\xi_j+y_2]\times[Mx_1,Mx_2]. \]
The vertical copy is a transposed copy of $R$. Let $C(P)$ be the family of all
rectangles $H_i(R)$ and $V_j(R)$, where $R\in P$ and $i,j\in[q]$. We call
this the crossbar operation. All these rectangles lie in $[0,L']^2$, where
$L'=ML+L$. The same spacing makes the row blocks pairwise disjoint and the
column blocks pairwise disjoint, so distinct horizontal copies do not meet one
another, and distinct vertical copies do not meet one another. Thus $C(P)$ has
$2q|P|$ rectangles.

For the base drawing, the first crossbar $P_1=C(P_0)$ is the 64-rectangle
family shown in \Cref{fig:p1-rectangles}.

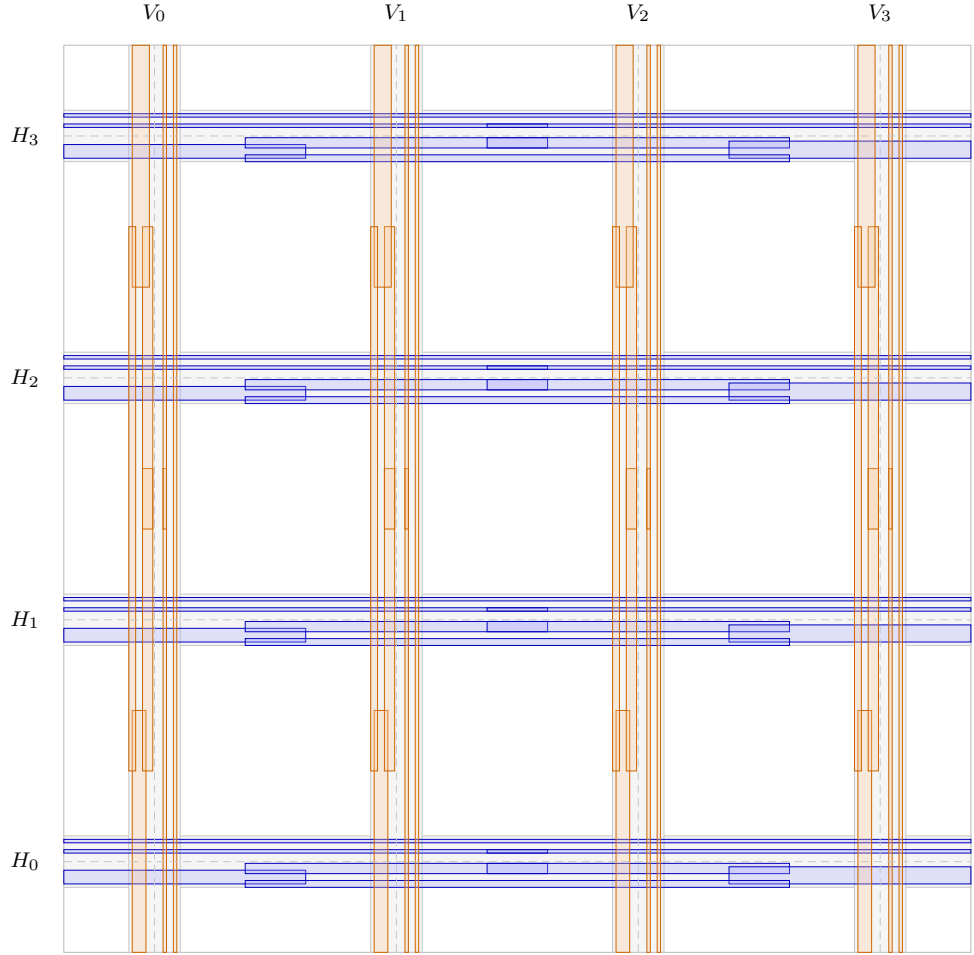
\begin{figure}[H]
\centering
\begin{tikzpicture}[
scale=1,
hrect/.style={fill=blue!35,draw=blue!75!black,line width=.22pt,fill opacity=.28},
vrect/.style={fill=orange!42,draw=orange!80!black,line width=.22pt,fill opacity=.28},
strip/.style={fill=gray!8,draw=gray!45,line width=.18pt},
marker/.style={gray!45,densely dashed,line width=.18pt},
copylabel/.style={font=\scriptsize}
]
\def\s{0.4}
\def\w{0.68}

\draw[gray!55,line width=.25pt] (0,0) rectangle (12,12);

\foreach \i/\xi in {0/3,1/11,2/19,3/27} {
  \pgfmathsetmacro{\c}{\s*\xi}
  \pgfmathsetmacro{\a}{\c-\w/2}
  \pgfmathsetmacro{\b}{\c+\w/2}
  \draw[strip] (0,\a) rectangle (12,\b);
  \node[copylabel,anchor=east] at (-.18,\c) {$H_{\i}$};
}

\foreach \j/\xi in {0/3,1/11,2/19,3/27} {
  \pgfmathsetmacro{\c}{\s*\xi}
  \pgfmathsetmacro{\a}{\c-\w/2}
  \pgfmathsetmacro{\b}{\c+\w/2}
  \draw[strip] (\a,0) rectangle (\b,12);
  \node[copylabel,anchor=south] at (\c,12.18) {$V_{\j}$};
}

\foreach \i/\xi in {0/3,1/11,2/19,3/27} {
  \pgfmathsetmacro{\row}{\s*\xi-\w/2}
  \foreach \xa/\xb/\ya/\yb in {
    0/30/26/28,
    0/16/20/22,
    14/30/20/22,
    0/8/2/10,
    6/16/8/14,
    14/24/8/14,
    22/30/2/12,
    6/24/0/4
  } {
    \pgfmathsetmacro{\xone}{\s*\xa}
    \pgfmathsetmacro{\xtwo}{\s*\xb}
    \pgfmathsetmacro{\yone}{\row+\w*\ya/30}
    \pgfmathsetmacro{\ytwo}{\row+\w*\yb/30}
    \filldraw[hrect] (\xone,\yone) rectangle (\xtwo,\ytwo);
  }
}

\foreach \j/\xi in {0/3,1/11,2/19,3/27} {
  \pgfmathsetmacro{\col}{\s*\xi-\w/2}
  \foreach \xa/\xb/\ya/\yb in {
    0/30/26/28,
    0/16/20/22,
    14/30/20/22,
    0/8/2/10,
    6/16/8/14,
    14/24/8/14,
    22/30/2/12,
    6/24/0/4
  } {
    \pgfmathsetmacro{\xone}{\col+\w*\ya/30}
    \pgfmathsetmacro{\xtwo}{\col+\w*\yb/30}
    \pgfmathsetmacro{\yone}{\s*\xa}
    \pgfmathsetmacro{\ytwo}{\s*\xb}
    \filldraw[vrect] (\xone,\yone) rectangle (\xtwo,\ytwo);
  }
}
\foreach \i/\xi in {0/3,1/11,2/19,3/27} {
  \pgfmathsetmacro{\c}{\s*\xi}
  \draw[marker] (0,\c) -- (12,\c);
}
\foreach \j/\xi in {0/3,1/11,2/19,3/27} {
  \pgfmathsetmacro{\c}{\s*\xi}
  \draw[marker] (\c,0) -- (\c,12);
}
\end{tikzpicture}
\caption{The 64 rectangles of $P_1=C(P_0)$. Blue rectangles are the four
horizontal copies of $P_0$, and orange rectangles are the four transposed
vertical copies. The long gaps have been compressed and the copied strips
enlarged to make all rectangles visible.}
\label{fig:p1-rectangles}
\end{figure}

\Cref{fig:p1-fat} gives a clearer drawing of the same 64-rectangle graph. It
uses compact coordinates that keep the horizontal and vertical copy structure
visible and leave room for the marker-grid overlay in
\Cref{fig:p1-piercing-overlay}. The formal recursion is still the coordinate
construction just defined.

\begin{figure}[H]
\centering
\begin{tikzpicture}[
scale=1.5, transform shape, 
x=.075cm,y=.075cm,
hrect/.style={fill=blue!35,draw=blue!75!black,line width=.22pt,fill opacity=.28},
vrect/.style={fill=orange!45,draw=orange!80!black,line width=.22pt,fill opacity=.28},
marker/.style={black!45,dash pattern=on 1.1pt off .9pt,line width=.28pt},
copylabel/.style={font=\scriptsize},
slotlabel/.style={font=\tiny,scale=.72,transform shape}
]
\draw[gray!55,line width=.25pt] (-1,-1) rectangle (111,111);

\foreach \i in {0,1,2,3} {
  \pgfmathsetmacro{\yc}{30*\i+10}
  \node[copylabel,anchor=east] at (-3,\yc) {$H_{\i}$};
  \foreach \xa/\xb/\ya/\yb in {
    0/11/9/10,
    0/6/7/8,
    5/11/7/8,
    0/3/1/3.2,
    2/6/2.8/6,
    5/9/2.8/6,
    8/11/1/6,
    2/9/0/2.4
  } {
    \pgfmathsetmacro{\xone}{10*\xa}
    \pgfmathsetmacro{\xtwo}{10*\xb}
    \pgfmathsetmacro{\yone}{30*\i+\ya+5}
    \pgfmathsetmacro{\ytwo}{30*\i+\yb+5}
    \filldraw[hrect] (\xone,\yone) rectangle (\xtwo,\ytwo);
  }
}

\foreach \j in {0,1,2,3} {
  \pgfmathsetmacro{\xc}{30*\j+10}
  \node[copylabel,anchor=south] at (\xc,113) {$V_{\j}$};
  \foreach \xa/\xb/\ya/\yb in {
    0/11/9/10,
    0/6/7/8,
    5/11/7/8,
    0/3/1/3.2,
    2/6/2.8/6,
    5/9/2.8/6,
    8/11/1/6,
    2/9/0/2.4
  } {
    \pgfmathsetmacro{\xone}{30*\j+\ya+5}
    \pgfmathsetmacro{\xtwo}{30*\j+\yb+5}
    \pgfmathsetmacro{\yone}{10*\xa}
    \pgfmathsetmacro{\ytwo}{10*\xb}
    \filldraw[vrect] (\xone,\yone) rectangle (\xtwo,\ytwo);
  }
}

\foreach \j in {0,1,2,3} {
  \foreach \t/\pos in {0/1.4,1/3.6,2/7.5,3/9.5} {
    \pgfmathsetmacro{\mx}{30*\j+5+\pos}
    \draw[marker] (\mx,-1) -- (\mx,111);
    \node[slotlabel,anchor=north] at (\mx,-1.25) {$\t$};
  }
}
\foreach \i in {0,1,2,3} {
  \foreach \b/\pos in {0/1.4,1/3.6,2/7.5,3/9.5} {
    \pgfmathsetmacro{\my}{30*\i+5+\pos}
    \draw[marker] (-1,\my) -- (111,\my);
    \node[slotlabel,anchor=west] at (111.15,\my) {$\b$};
  }
}
\node[slotlabel,anchor=north east] at (-1.15,-1.25) {$t$};
\node[slotlabel,anchor=west] at (111.15,112) {$b$};
\end{tikzpicture}
\caption{A fatter drawing of the same graph as \Cref{fig:p1-rectangles}. The
compact coordinates keep the $H_i$ and $V_j$ copies visible; the dashed lines
show the marker grid used again in \Cref{fig:p1-piercing-overlay}. In column
$V_j$, the labels below the drawing give $t$, so the vertical marker lines are named
$\xi'_{(j,t)}$; in row $H_i$, the right labels give $b$, so the horizontal
marker lines are named $\eta'_{(i,b)}$.}
\label{fig:p1-fat}
\end{figure}
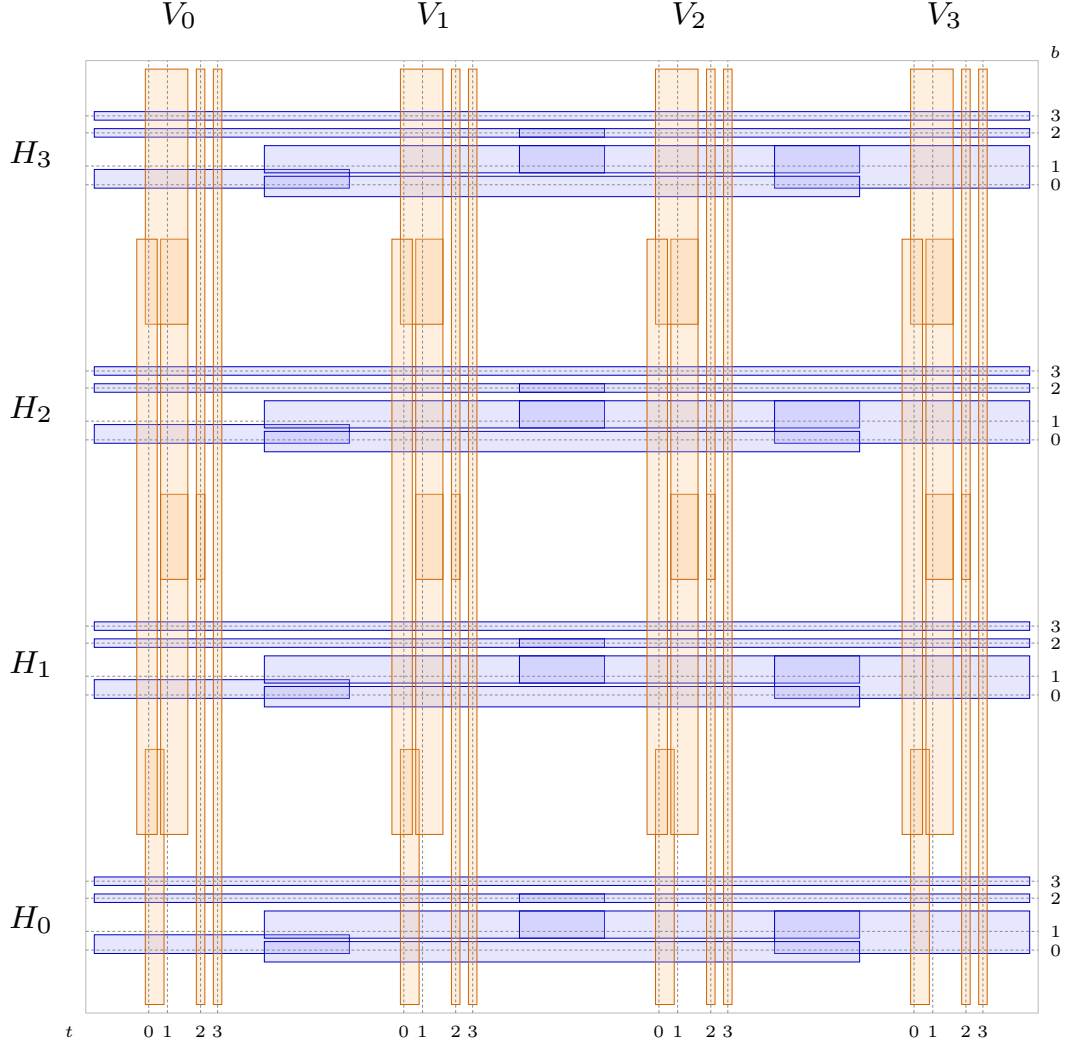

For the coordinate construction defining $C(P)$, the reason for the scaling is
that each old marker used to index a row or column has been expanded to a whole
block of length $L$ in the relevant coordinate. This makes the crossing rule
exact, and we prove it formally as a lemma.

\begin{lemma}[Crossing rule]\label{lem:crossing-rule}
For rectangles $R,S\in P$ and slots $i,j\in[q]$,
\[ H_i(R)\cap V_j(S)\ne\emptyset
\quad\Longleftrightarrow\quad
j\in X(R)\text{ and }i\in X(S). \]
\end{lemma}

\begin{proof}
The $x$-interval of $V_j(S)$ lies inside the block
$[M\xi_j,M\xi_j+L]$. If $j\in X(R)$, then $\xi_j$ lies at least one unit
inside the $x$-interval of $R$, so $[Mx_1,Mx_2]$ contains the whole block
$[M\xi_j,M\xi_j+L]$. If $j\notin X(R)$, the same separation makes
$[Mx_1,Mx_2]$ disjoint from that block. Thus the two new rectangles overlap
in the $x$-coordinate exactly when $j\in X(R)$.

The $y$-coordinate is the same argument with $S$ in place of $R$: the
$y$-interval of $H_i(R)$ lies inside the block
$[M\xi_i,M\xi_i+L]$, while the $y$-interval of $V_j(S)$ is the scaled
$x$-interval of $S$. They overlap exactly when $i\in X(S)$.
Together, the two coordinates carry the two slot tests separately, giving the
claimed equivalence.
\end{proof}

The new drawing also has marker lines. Each old marker has been expanded to a
block that contains a copy of the old drawing, so a new marker is labelled by
two pieces of data: the old block, and the old marker inside that block. Index
the new markers by ordered pairs in $[q]\times[q]$ and set
\[ \xi'_{(j,t)}=M\xi_j+\eta_t,\qquad
\eta'_{(i,t)}=M\xi_i+\eta_t. \]
They are ordered lexicographically by the first coordinate.

The formulas for the new lists now come directly from the coordinates. The
rectangle $H_i(R)$ is stretched in the $x$-direction, so whenever $R$ contained
the old vertical marker $j$, the copy $H_i(R)$ contains the whole new block
over $j$, and hence all markers $(j,t)$. This gives
$X'(H_i(R))=X(R)\times[q]$. In the $y$-direction, the same rectangle lives only
in row $i$, and inside that row it sees the old horizontal markers of $R$; this
gives $Y'(H_i(R))=\{i\}\times Y(R)$. The vertical copies are transposed, so the
same reading swaps the old horizontal and vertical directions. Altogether,
\[
\begin{aligned}
X'(H_i(R))&=X(R)\times[q],&
Y'(H_i(R))&=\{i\}\times Y(R),\\
X'(V_j(R))&=\{j\}\times Y(R),&
Y'(V_j(R))&=X(R)\times[q].
\end{aligned}
\]
The entry $X'(V_j(R))=\{j\}\times Y(R)$ is the point where old horizontal
data becomes new vertical data. This is the precise place where the second
list system is used.

The same one-unit separation is preserved: after scaling by $M=L+1$, every
old one-unit margin becomes a gap of length at least $L+1$, and inserting a
copy of the old $[0,L]^2$ drawing still leaves room on the required side of
each new marker. The marker spacing is preserved as well. Inside one old block
it is the old spacing among the markers $\eta_t$; between consecutive old
blocks it is at least $M(\xi_{j+1}-\xi_j)-L\ge 1$. Thus the construction can be
repeated\footnote{For the first step, this bookkeeping has a very nice and concrete graph-level picture, and we record that in
\Cref{app:first-crossbar-schematics}, which shows the sixteen crossings of $P_1$ and the complete bipartite pieces that appear inside them.}.

We now show that the repetition preserves the Hall properties.

\begin{lemma}\label{lem:repetition-hall}
If both $(P,X)$ and $(P,Y)$ have the Hall property, then both
$(C(P),X')$ and $(C(P),Y')$ have the Hall property.
\end{lemma}

\begin{proof}
We prove the statement for $X'$. Let $I$ be an independent set in $C(P)$.
First consider a fixed vertical copy $V_j(P)$. Since $(P,Y)$ has the Hall
property, the rectangles of $I\cap V_j(P)$ can be assigned distinct old
$Y$-slots. If $V_j(S)$ receives the old slot $t\in Y(S)$, place it in the new
slot $(j,t)$. Here $Y$ is used to choose distinct inner positions in the
column $j$.

Now consider a fixed horizontal copy $H_i(P)$. Since $(P,X)$ has the Hall
property, the rectangles of $I\cap H_i(P)$ can be assigned distinct old
$X$-slots. Think of such an assignment as choosing an outer column $j$ for
each selected horizontal rectangle.

Fix a column $j$. Let
\[ A_j=\bigcup_{V_j(S)\in I} X(S). \]
The same vertical rectangles also block rows, and this blocking is measured
by their old $X$-lists. The rectangles of $I\cap V_j(P)$ are independent
inside the old graph, so the Hall property for $X$ gives
$|I\cap V_j(P)|\le |A_j|$. If $i\in A_j$, then some selected $V_j(S)$ has
$i\in X(S)$. By \Cref{lem:crossing-rule}, no selected horizontal rectangle in
row $i$ can be assigned to column $j$, since that assignment would mean
$j\in X(R)$ and would create an intersection.

Thus the selected horizontal rectangles assigned to column $j$ can come only
from the $q-|A_j|$ rows outside $A_j$. The old $X$-assignment is injective
inside each horizontal copy, so there is at most one such rectangle from each
row. There are already $|I\cap V_j(P)|$ occupied slots among the $q$ slots
$(j,0),\ldots,(j,q-1)$, so at least $q-|I\cap V_j(P)|$ remain free. Since
$q-|A_j|\le q-|I\cap V_j(P)|$, the horizontal rectangles assigned to column
$j$ can be placed into these free inner slots. This is allowed because
$X'(H_i(R))=X(R)\times[q]$: once $H_i(R)$ has been assigned to column $j$, it
may use any inner slot $(j,t)$. Doing this for every $j$ injects all of $I$
into the $X'$-slots.

The proof for $Y'$ is the same with the roles of horizontal and vertical
copies interchanged; the formulas for $Y'$ are obtained from the formulas for
$X'$ by making exactly that coordinate swap. Equivalently, the column-by-column
assignment above becomes a row-by-row assignment, with old $X$- and $Y$-slots
interchanged.
\end{proof}

Starting with the base drawing $P_0$ from \Cref{prop:base-realization}, define
\[ P_{r+1}=C(P_r). \]
Let $q_r$ be the number of slots in each direction for $P_r$. Then
\[ q_0=4,\qquad q_{r+1}=q_r^2,\qquad q_r=4^{2^r}. \]
By \Cref{lem:repetition-hall}, every $P_r$ has the Hall property in both
directions.

\begin{proposition}\label{prop:independence-recursion}
For every $r\ge 0$, $\alpha(P_r)=q_r$ and $|P_r|=2^{r+1}q_r$.
\end{proposition}

\begin{proof}
The Hall property gives $\alpha(P_r)\le q_r$. For the reverse inequality,
notice that $P_{r+1}$ contains $q_r$ pairwise disjoint horizontal copies of
$P_r$. Taking a maximum independent set in each of these copies gives
$\alpha(P_{r+1})\ge q_r\alpha(P_r)$. Since
$\alpha(P_0)=\alpha(F)=4=q_0$, induction gives $\alpha(P_r)=q_r$ for all
$r$.

The size formula follows from $|P_0|=8=2q_0$ and
$|P_{r+1}|=2q_r|P_r|$.
\end{proof}

\begin{proof}[Proof of \Cref{thm:counterexample}]
The first repeated drawing $P_1$ has
\[ |P_1|=2\cdot 4\cdot 8=64 \]
rectangles. By \Cref{prop:independence-recursion},
$\nu(P_1)=\alpha(P_1)=q_1=16$.

It remains to check triangle-freeness. A triangle cannot lie inside one copy,
because $P_0$ is the triangle-free graph $C_5\sqcup K_2\sqcup K_1$.
Distinct horizontal copies are disjoint from one another, and distinct
vertical copies are disjoint from one another. Hence any remaining triangle
would have two rectangles from one copy and one rectangle from a copy in the
other direction.

Suppose two rectangles $H_i(R_1),H_i(R_2)$ in the same horizontal copy both
meet a rectangle $V_j(S)$. By \Cref{lem:crossing-rule}, both old rectangles
$R_1$ and $R_2$ have $j$ in their $X$-lists. By \Cref{lem:base-x}, they are
nonadjacent in the base graph, so their horizontal copies do not intersect.
The case with two rectangles in one vertical copy is identical. Therefore
$P_1$ is triangle-free.

Since no point can lie in three rectangles of a triangle-free rectangle
family, every piercing point meets at most two rectangles. Thus
\[ \tau(P_1)\ge \frac{64}{2}=32. \]
But $2\nu(P_1)-1=31$, so $\tau(P_1)>2\nu(P_1)-1$.
The exact value of $\tau(P_1)$ is not needed for this conclusion; the lower
bound $32$ already crosses the conjectured threshold.

This is the only place where the repeated family is used through
triangle-freeness. The later levels retain the recursive slot data and the
identity $\nu(P_r)=q_r$, which are the features used for the LP and piercing
results below.
\end{proof}

\section{Finite LP gaps}\label{sec:finite-lp-gaps}

The same separation into local copies and mixed crossings also makes the LP
checks small. We work with the clique relaxation from
\Cref{sec:overview}; for rectangle families, this is the standard point
relaxation. The guiding idea is to judge a weighting on one copy by two
features at once: the total value it contributes, and the amount of clique
weight that can be seen through each slot when another copy crosses it. The
latter quantities will be defined below as loads. A high-value weighting is
useful only if its exposed loads can be paired with small enough loads from the
copy crossing it.

The first repeated family $P_1$ already gives a feasible LP value equal to
$2\nu(P_1)$: since $P_1$ is triangle-free, assigning weight $1/2$ to every
rectangle is feasible and has value $32=2\nu(P_1)$. At the next level there is
a short hand-checkable weighting whose value is strictly larger than
$2\nu(P_2)$. We describe this weighting before the larger $P_3$ construction,
because it shows the mechanism in its smallest form.

We first set up the bookkeeping used to check feasibility when weights are
placed on horizontal and vertical copies.

For a nonnegative weight vector $w$ on a rectangle family $P$, write $w(K)$
for the sum of the weights on a clique $K$. If $P$ has $x$-slots and
$y$-slots, the rectangles using a fixed $x$-slot $s$ form the $x$-slot class
of $s$. Define
\[
X_w(s)=\max\{w(K):K\text{ is a clique and every rectangle in }K
\text{ uses }x\text{-slot }s\},
\]
\[
Y_w(t)=\max\{w(K):K\text{ is a clique and every rectangle in }K
\text{ uses }y\text{-slot }t\},
\]
and
\[
Z_w(s,t)=\max\{w(K):K\text{ is a clique and every rectangle in }K
\text{ uses both slots }s,t\}.
\]
We call these numbers the $x$-load at $s$, the $y$-load at $t$, and the
joint load at $(s,t)$, respectively. The empty clique is allowed in these
maxima and has weight zero.

\begin{lemma}\label{lem:lp-crossbar-loads}
Let $P$ be one of the recursive rectangle families, and put weights $h_i$ on
the horizontal copies $H_i(P)$ and weights $v_j$ on the vertical copies
$V_j(P)$ inside $C(P)$. Assume that each $h_i$ and $v_j$ is feasible on its
own copy. If
\[ X_{h_i}(j)+X_{v_j}(i)\le 1 \]
for every pair $i,j$, then the resulting weight vector on $C(P)$ is
clique-feasible.

Moreover, if $X'$ denotes the $x$-load function of the resulting weight vector,
then for the new $x$-slot $(j,t)$ in $C(P)$,
\[
X'(j,t)\le
\max\left\{
Y_{v_j}(t),
\max_i\bigl(X_{h_i}(j)+Z_{v_j}(i,t)\bigr)
\right\}.
\]
\end{lemma}

\begin{proof}
Distinct horizontal copies are disjoint from one another, and distinct vertical
copies are disjoint from one another. Hence every clique in $C(P)$ is either
contained in one copy, or lies in $H_i(P)\cup V_j(P)$ for a unique pair $i,j$.
The local cliques are feasible by assumption. In the mixed case,
\Cref{lem:crossing-rule} says that the
horizontal part is contained in the old $x$-slot class for $j$, and the
vertical part is contained in the old $x$-slot class for $i$. Thus its weight
is at most
$X_{h_i}(j)+X_{v_j}(i)$.

For the displayed bound, a clique in the new $x$-slot $(j,t)$ is either
wholly inside $V_j(P)$, giving the term $Y_{v_j}(t)$, or has a horizontal part
inside some $H_i(P)$. In the second case its horizontal part lies in the old
$x$-slot class for $j$. Its vertical part, if present, lies in the old joint
slot class for $(i,t)$; if it is absent, the empty clique contributes zero to
$Z_{v_j}(i,t)$.
\end{proof}

The lemma also suggests a useful compression. In $P_1$, the $x$-slots are
pairs $(a,b)\in[4]^2$. The first coordinate records the old block, while the
second coordinate records which base slot lies inside that block. The
weightings below treat the old blocks uniformly, so their loads depend only on
the second coordinate. We call this coordinate the final digit of the slot.
In $P_2=C(P_1)$, the horizontal and vertical copies of $P_1$ are indexed by
the sixteen slots of $P_1$, and the final digit of a copy means the final digit
of its index.

This convention keeps the finite checks small. Although there are sixteen
copies of $P_1$ in each direction inside $P_2$, the weightings below have only
four types, according to this final digit.

The two subsections below have different roles. The first gives a direct
feasible weighting on $P_2$ with value larger than $2\nu(P_2)$. The second
constructs a more evenly distributed vector on $P_2$, whose $x$-loads are all
at most $1/2$; placing that vector in every copy at the next level gives the
$P_3$ gap from \Cref{thm:p3-gap}.
The one-copy load checks used in both subsections are collected in
\Cref{app:profile-checks}.

\subsection{A first gap above 2}\label{subsec:first-gap-above-two}

We begin with a feasible vector on $P_2$ whose value is already larger than
$2\nu(P_2)$. The ingredients are six weightings of $P_1$.

Let $\mathcal C=\{C_0,C_1,C_2,C_3,Q\}$ be the five-cycle in the base graph,
and let $\mathcal U=\{U_L,U_R\}$. Inside a copy of $P_1$, there are four
horizontal copies of $P_0$ and four vertical copies of $P_0$. We use a small
class of weightings chosen to follow the distinctions made by the final-digit
checks. The horizontal half contributes in the same way to every final digit. In
the vertical half, digits $0$ and $1$ see the five-cycle part, digit $2$ sees
the pair $U_L,U_R$, and digit $3$ sees $T$. This suggests four weights: let
$h$ be the weight on every rectangle in the horizontal half of $P_1$; in the
vertical half, let $c$ be the weight on the five-cycle vertices $\mathcal C$,
let $u$ be the weight on the two vertices in $\mathcal U$, and let $t$ be the
weight on $T$.

The four horizontal copies contribute $4\cdot 8h$, while the four vertical
copies contribute $4\cdot 5c+4\cdot 2u+4t$. Thus such a weighting has total value
\[ 32h+20c+8u+4t. \]

With these four weights, the relevant $x$-loads are easy to read. For a
slot $(a,b)$ of $P_1$, only the final digit $b$ matters. If $b=0$ or $1$, the
vertical vertices using that old horizontal slot form a path of two possible
edges in the five-cycle. A clique using the slot either lies wholly in that
vertical copy, giving weight at most $2c$, or also uses one horizontal
rectangle, giving weight at most $h+c$. Thus
\[
X(0)=X(1)=\max\{h+c,2c\}.
\]
The same reasoning for the pair $U_L,U_R$ gives
\[
X(2)=\max\{h+u,2u\},
\]
and the slot containing only $T$ gives
\[
X(3)=h+t.
\]

So the choice of profiles becomes a small load-budget problem. When a
horizontal copy of final digit $b$ faces a vertical copy of final digit $d$ in
$P_2$, \Cref{lem:lp-crossbar-loads} asks that the $d$-load of the horizontal
profile and the $b$-load of the vertical profile add to at most one.

The table below can be found by working backwards from this budget. The first
useful place to look is the high-value end of this four-parameter class. The
uniform horizontal weight may be set to $h=1/2$, and cannot be made larger,
because each horizontal copy contains edges. The same reasoning lets us set the
five-cycle weight to $c=1/2$. We keep $t=0$, since $T$ spends the
final-digit-$3$ budget and contributes only four rectangles. The remaining
useful dial is $u$, which changes the final-digit-$2$ load. This leads to the
one-parameter family
\[
(h,c,u,t)=(1/2,1/2,u,0).
\]
For $0\le u\le 1/4$, which is the range used below, this family has value
$26+8u$ and load vector
\[
(X(0),X(1),X(2),X(3))=(1,1,1/2+u,1/2).
\]
Thus it spends the full budget in final digits $0$ and $1$, spends only half
the budget in final digit $3$, and leaves a tunable amount in final digit $2$.
Trying to use this family in the vertical copies of $P_2$ immediately forces
the horizontal rows with final digits $0$ and $1$ to be zero. The row with
final digit $3$ may carry load $1/2$ in every digit, giving the balanced
profile $B$ below. The row with final digit $2$ is where the remaining freedom
lies.

The row with final digit $2$ must now mediate between these choices. In two
columns it should use half the budget, so only the choice $u=0$ from the
high-value family fits opposite it. In one column it should use little budget,
so the largest choice from the same family fits. In the remaining column it
should use an intermediate amount. The first two requirements set the first
two entries to $1/2$, and the largest choice $u=1/4$ sets the last entry to
$1/4$. The middle entry may be any value between $1/4$ and $1/2$; changing it
only shifts digit-$2$ weight from the row profile to the middle vertical
profile. We take the midpoint, giving the target
\[
\ell=(1/2,1/2,3/8,1/4).
\]
Pushing the four weights under these caps gives
$(h,c,u,t)=(1/4,1/4,1/8,0)$: the cap in final digit $3$ permits
$h=1/4$ with no weight on $T$; then the caps in final digits $0,1$ permit
$c=1/4$, and the cap in final digit $2$ permits $u=1/8$. This is the profile
$A$. Now a vertical column of final digit $d$ has remaining budget
$1-\ell_d$, where $\ell_d$ is the $d$th entry of $\ell$. Since the high-value
family contributes $1/2+u$ in row $2$, we need $u\le 1/2-\ell_d$. Hence
digits $0$ and $1$ force $u=0$, digit $2$ allows $u=1/8$, and digit $3$
allows $u=1/4$. These are the profiles $C,D,E$.

The six profiles below record this solution. A bare number in the table means
that every rectangle in that half receives that weight; omitted parts receive
weight zero.

\begin{equation}\label{eq:p2-profiles}
\begin{array}{c|c|c|c|c}
\text{name}&\text{horizontal half}&\text{vertical half}&
\text{value}&(X(0),X(1),X(2),X(3))\\ \hline
O&0&0&0&(0,0,0,0)\\
A&1/4&\mathcal C:1/4,\ \mathcal U:1/8,\ T:0&
14&(1/2,1/2,3/8,1/4)\\
B&3/8&1/8&16&(1/2,1/2,1/2,1/2)\\
C&1/2&\mathcal C:1/2,\ \mathcal U:0,\ T:0&
26&(1,1,1/2,1/2)\\
D&1/2&\mathcal C:1/2,\ \mathcal U:1/8,\ T:0&
27&(1,1,5/8,1/2)\\
E&1/2&\mathcal C:1/2,\ \mathcal U:1/4,\ T:0&
28&(1,1,3/4,1/2).
\end{array}
\end{equation}

In the last column, $X(b)$ denotes an upper bound on $X_w(a,b)$ for any first
digit $a$; the bound depends only on the final digit $b$. For the displayed
profiles, these bounds are exactly the numbers given by the formulas above.
For instance, profile $A$ has
$(h,c,u,t)=(1/4,1/4,1/8,0)$, and hence
$X(2)=h+u=3/8$. Profile $E$ has
$(h,c,u,t)=(1/2,1/2,1/4,0)$, and hence
$X(2)=h+u=3/4$.

Each of the six weightings is feasible on $P_1$.
Feasibility on $P_1$ is immediate from triangle-freeness: every clique is a
vertex or an edge, and every edge in the listed supports has total weight at
most one. The mixed table below shows how the local bounds spend the available
slack.

In \Cref{fig:p2-lp-bookkeeping}, $b$ denotes the final digit of the horizontal
copy index and $d$ denotes the final digit of the vertical copy index.

\begin{figure}[H]
\centering
\begin{tikzpicture}[
font=\small,
panel/.style={anchor=north,align=center,inner sep=0pt}
]
\node[panel] at (0,0) {
profile assignment\\[3pt]
$\begin{array}{c|cccc}
&0&1&2&3\\ \hline
\text{horizontal}&O&O&A&B\\
\text{vertical}&C&C&D&E
\end{array}$
};
\node[panel] at (6.2,0) {
	mixed clique sums\\[3pt]
	$\begin{array}{c|ccc}
	&d=0,1&d=2&d=3\\ \hline
	b=0,1&0+1&0+1&0+1\\
	b=2&1/2+1/2&3/8+5/8&1/4+3/4\\
	b=3&1/2+1/2&1/2+1/2&1/2+1/2
	\end{array}$
	};
\end{tikzpicture}
\caption{Bookkeeping for the $P_2$ weighting. The final digits are the four
base slot types, and the right panel records how the chosen profiles spend the
mixed-clique load budget in \Cref{prop:p2-gap}.}
\label{fig:p2-lp-bookkeeping}
\end{figure}

We now show that $P_2$ actually beats $2$.

\begin{proposition}\label{prop:p2-gap}
The second recursive family satisfies
\[ \frac{\alpha^*(P_2)}{\nu(P_2)}\ge\frac{137}{64}>2. \]
\end{proposition}

\begin{proof}
In a horizontal copy $H_i(P_1)$ of $P_2$, place the weightings
\[
O,\ O,\ A,\ B
\]
when the final digit of $i$ is $0,1,2,3$, respectively. In a vertical copy
$V_j(P_1)$, place
\[
C,\ C,\ D,\ E
\]
when the final digit of $j$ is $0,1,2,3$, respectively.

Local cliques are feasible by the discussion above. For a mixed clique
between $H_i(P_1)$ and $V_j(P_1)$, let $b$ be the final digit of $i$ and $d$
the final digit of $j$. By \Cref{lem:lp-crossbar-loads}, it is enough to
check the following sums:
\begin{equation}\label{eq:p2-mixed-loads}
\begin{array}{c|ccc}
&d=0,1&d=2&d=3\\ \hline
b=0,1&0+1&0+1&0+1\\
b=2&1/2+1/2&3/8+5/8&1/4+3/4\\
b=3&1/2+1/2&1/2+1/2&1/2+1/2.
\end{array}
\end{equation}
Every entry is at most one, so the whole weight vector is clique-feasible.

Each final digit occurs four times among the sixteen copies. The total weight
is therefore
\[
4(14)+4(16)+8(26)+4(27)+4(28)
=548.
\]
By \Cref{prop:independence-recursion}, $\nu(P_2)=q_2=256$. Hence
\[ \frac{\alpha^*(P_2)}{\nu(P_2)}\ge\frac{548}{256}=\frac{137}{64}. \]
\end{proof}

\subsection{A finite gap above the previous bound}
\label{subsec:gap-above-previous-bound}

For the next finite improvement, we need a vector that can be placed in every
horizontal and every vertical copy of $P_2$ inside $P_3$. The previous
weighting beats $2$ by matching heavier profiles with lighter profiles in
particular final digits. That is enough for one level. A reusable vector needs
a uniform cap instead: every old $x$-slot class should contribute at most
$1/2$, so that a mixed clique in $P_3$ sees at most $1/2+1/2$.

The same level $P_2$ contains such a vector. Its total value is smaller than
the one in the previous subsection, but the uniform $1/2$ bound on its
$x$-loads makes it reusable.

We use four simpler weightings on $P_1$, together with the zero profile, which
we denote by $0$; again, omitted parts receive weight zero. The profiles $R,S$
are used in horizontal copies, while $U,V$ are used in vertical copies; the
table is arranged so that the mixed clique sums and the new $x$-loads can both
be checked from the final digit:
\begin{equation}\label{eq:p3-profiles}
\begin{array}{c|c|c|c}
\text{name}&\text{definition}&\text{value}&(X(0),X(1),X(2),X(3))\\ \hline
R&\text{horizontal half: }1/4;\ \text{vertical }\mathcal C:1/4&13&
(1/2,1/2,1/4,1/4)\\
S&\text{horizontal half: }1/2;\ \text{vertical half: }0&16&
(1/2,1/2,1/2,1/2)\\
U&\text{vertical }\mathcal C:1/2;\ \text{all else: }0&10&
(1,1,0,0)\\
V&\text{vertical }\mathcal C:1/2,\ \mathcal U:1/4,\ T:0;\ \text{horizontal half: }0&
12&(1,1,1/2,0).
\end{array}
\end{equation}
For $U$ and $V$ we shall also use
\begin{equation}\label{eq:p3-extra-loads}
Y_U(t),Y_V(t)\le\frac12,\qquad
Z_U\le(1/2,1/2,0,0),\qquad
Z_V\le(1/2,1/2,1/4,0),
\end{equation}
where the entries of $Z_U$ and $Z_V$ are again indexed by the final digit of
the $x$-slot, uniformly in the $y$-slot. These bounds follow from the same
one-copy check as before, recorded in \Cref{app:profile-checks}. In a vertical
copy, fixing a $y$-slot fixes an old $X$-slot in the base, because
$Y'(V_j(R))=X(R)\times[4]$; vertices using one old $X$-slot are independent.

In \Cref{fig:a2-lp-bookkeeping}, $b$ denotes the final digit of the horizontal
copy index and $d$ denotes the final digit of the vertical copy index.

\begin{figure}[H]
\centering
\begin{tikzpicture}[
font=\small,
panel/.style={anchor=north,align=center,inner sep=0pt}
]
\node[panel] at (0,0) {
profile assignment\\[3pt]
$\begin{array}{c|cccc}
&0&1&2&3\\ \hline
\text{horizontal}&0&0&R&S\\
\text{vertical}&U&U&V&V
\end{array}$
};
\node[panel] at (5.05,0) {
mixed clique checks\\[3pt]
$\begin{array}{c|cc}
&d=0,1&d=2,3\\ \hline
b=0,1&1&1\\
b=2&1/2&3/4\\
b=3&1/2&1/2
\end{array}$
};
\node[panel] at (9.75,0) {
new $x$-load checks\\[3pt]
$\begin{array}{c|cc}
&V_j=U&V_j=V\\ \hline
b=0,1&\le1/2&\le1/2\\
b=2&\le1/2&\le1/2\\
b=3&\le1/2&\le1/2
\end{array}$
};
\end{tikzpicture}
\caption{Bookkeeping for the reusable vector $A_2$. The profile assignment is
shown on the left. The two checks record the mixed-clique bounds inside $P_2$
and the uniform $1/2$ bound on the new $x$-loads.}
\label{fig:a2-lp-bookkeeping}
\end{figure}

\begin{lemma}\label{lem:a2-half-load}
There is a clique-feasible weight vector $A_2$ on $P_2$ with total value
$292$ such that $X_{A_2}(s)\le 1/2$ for every $x$-slot $s$ of $P_2$.
\end{lemma}

\begin{proof}
In the horizontal copies of $P_1$ inside $P_2$, place
\[
0,\ 0,\ R,\ S
\]
according to final digit $0,1,2,3$. In the vertical copies, place
\[
U,\ U,\ V,\ V.
\]
The value is
\[
4(13)+4(16)+8(10)+8(12)=292.
\]

Feasibility follows from \Cref{lem:lp-crossbar-loads} and the table
\begin{equation}\label{eq:a2-mixed-loads}
\begin{array}{c|cc}
&d=0,1&d=2,3\\ \hline
b=0,1&1&1\\
b=2&1/2&3/4\\
b=3&1/2&1/2
\end{array}
\end{equation}
where $b$ is the final digit of the horizontal copy and $d$ is the final
digit of the vertical copy. For example, the entry $3/4$ is
$X_R(2)+X_V(2)\le 1/4+1/2$.

It remains to bound the new $x$-loads. Fix a new $x$-slot $(j,t)$ of
$P_2$. If the clique is wholly inside the vertical copy $V_j(P_1)$, the load
is at most $1/2$ by \eqref{eq:p3-extra-loads}. If the clique has a horizontal
part, the second part of \Cref{lem:lp-crossbar-loads} and the bounds above give
\begin{equation}\label{eq:a2-x-loads}
\begin{array}{c|cc}
&V_j=U&V_j=V\\ \hline
b=0,1&\le1/2&\le1/2\\
b=2&1/2&1/2\\
b=3&1/2&1/2.
\end{array}
\end{equation}
Here the entries are the resulting upper bounds. For $b=0,1$ the horizontal
term is zero and the vertical term is at most $1/2$ by
\eqref{eq:p3-extra-loads}; the remaining entries come from
$1/2+0$, $1/4+1/4$, $1/2+0$, and $1/2+0$. Every entry is at most $1/2$.
\end{proof}

\begin{proof}[Proof of \Cref{thm:p3-gap}]
Form $P_3=C(P_2)$, and put one copy of the vector $A_2$ from
\Cref{lem:a2-half-load} in every horizontal and every vertical copy of
$P_2$. Local cliques are feasible. A mixed clique meets one horizontal copy
and one vertical copy, and by \Cref{lem:a2-half-load} each side contributes
at most $1/2$ to the relevant old $x$-slot class. Thus every mixed clique has
weight at most one.

The total value is
\[ 2q_2\cdot 292=2\cdot256\cdot292=149504. \]
Since $\nu(P_3)=q_3=65536$,
\[
\frac{\alpha^*(P_3)}{\nu(P_3)}
\ge\frac{149504}{65536}
=\frac{73}{32}.
\]
Finally, compared with the bound of
\cite[Theorem~15]{ajwani2026counterexamplewegnersconjectureaxisparallel},
\[ \frac{73}{32}-\frac{17891}{8064}=\frac{505}{8064}>0. \]
\end{proof}

\section{Fractional solutions approaching \texorpdfstring{$5/2$}{5/2}}\label{sec:recursive-lower}

The vector $A_2$ from \Cref{lem:a2-half-load} leaves us with a useful clue.
Its total value gives the finite gap at $P_3$, but the more important feature
is the uniform bound $X_{A_2}(s)\le 1/2$. Once a weighting has this half-load
property, it can be placed in every horizontal and every vertical copy at the
next level: every mixed clique then sees at most $1/2$ from each side.

This suggests the question for the rest of the section. Can the half-load
property be propagated while the normalized value keeps growing? If we can
build, on each $P_r$, a half-load weighting whose normalized value tends to
$5/4$, then putting it in every copy of $P_{r+1}$ gives ratios tending to
$5/2$.

At first it may seem that such a recursion has to remember every slot at every
level. The construction has a simpler trace. In the base slots, split
\[ \mathsf L=\{0,1\},\qquad \mathsf H=\{2,3\}. \]
For a slot of $P_{r+1}$, written as an ordered pair $(j,t)$ of old slots, its
class is the class of the inner slot $t$. Thus every level has two equally
large slot classes, low and high. The finite constructions above used the
whole final digit; the limiting construction only needs this low/high
information.

This split is already visible in the profile that starts the recursion. The
five-cycle vertices of $P_0$ use only the horizontal slots $0$ and $1$. When
those vertices are placed in vertical copies, old horizontal slots become new
$x$-slots. Thus the first source of fractional weight naturally lives on the
low $x$-side and leaves the high $x$-side empty. The recursion keeps exploiting
this simple imbalance: low slots carry weight, high slots provide room, and
the next crossing only needs to know which of these two situations it is
seeing.

For a weighting $w$ on $P_r$, write $X_w(\mathsf L)$ for the maximum of
$X_w(s)$ over low $x$-slots $s$, and define $X_w(\mathsf H)$ similarly. We
use the same convention for $Y_w$. For joint loads, $Z_w(\mathsf L,\mathsf H)$
means the maximum of $Z_w(s,t)$ over low $x$-slots $s$ and high $y$-slots $t$.
Rows of a displayed $Z$-matrix are indexed by the $x$-class, and columns by
the $y$-class.

The next lemma records the old crossing calculation after this compression.
It is the only formal bookkeeping needed in the section: once the loads are
known in the two classes, the next level can be checked without remembering
the individual slots.

\begin{lemma}\label{lem:two-class-loads}
Let $P=P_r$. Suppose that $h_{\mathsf L},h_{\mathsf H}$ and
$v_{\mathsf L},v_{\mathsf H}$ are feasible weightings on $P$. In $C(P)$,
place $h_b$ in each horizontal copy whose index has class $b$, and place
$v_d$ in each vertical copy whose index has class $d$. If
\[ X_{h_b}(d)+X_{v_d}(b)\le 1 \]
for all $b,d\in\{\mathsf L,\mathsf H\}$, then the resulting weighting on
$C(P)$ is feasible.

Moreover, its new class loads satisfy
\[
X'_f\le
\max_{b,d}\{Y_{v_d}(f),\,X_{h_b}(d)+Z_{v_d}(b,f)\},
\]
\[
Y'_e\le
\max_{b,d}\{Y_{h_b}(e),\,X_{v_d}(b)+Z_{h_b}(d,e)\},
\]
and
\[
Z'_{f,e}\le
\max_{b,d}\{Z_{h_b}(d,e)+Z_{v_d}(b,f)\},
\]
where $b,d,e,f\in\{\mathsf L,\mathsf H\}$. Here $X'_f$ is the maximum new
$x$-load over slots of class $f$, and $Y'_e$ and $Z'_{f,e}$ are interpreted
similarly.
\end{lemma}

\begin{proof}
The feasibility statement is \Cref{lem:lp-crossbar-loads}, after taking the
maximum over all slots in the relevant classes.

The bound for $X'$ is also the bound from \Cref{lem:lp-crossbar-loads}, again
maximized over the two slot classes. The bound for $Y'$ is the same argument
with the two coordinates interchanged. For the joint load, fix a new $x$-slot
$(j,t)$ and a new $y$-slot $(i,u)$. A clique using both slots has a horizontal
part using the old joint slot class $(j,u)$ and a vertical part using the old
joint slot class $(i,t)$. Empty parts contribute zero, so the displayed upper
bound follows after passing to the classes of $i,j,t,u$.
\end{proof}

Since $P_r$ is obtained by a horizontal-vertical construction, every $P_r$ with
$r\ge 1$ is invariant under swapping the two coordinates. If $w$ is a
weighting on such a $P_r$, let $w^T$ denote the weighting obtained by this
swap. The coordinate swap sends horizontal copies to vertical copies and
interchanges the two marker systems. Thus $w^T$ has the same total value; its
$X$- and $Y$-loads are interchanged, and its $Z$-matrix is transposed. We will
also use scalar multiples of
weightings, with all weights and all loads scaled by the same factor.

We now describe the four jobs that the recursion keeps separate. A profile has
normalized value $a$ if its total weight is $a q_r$. The table below should be
read as a list of promises a weighting makes to the next round: how much total
value it carries, and how much clique weight can be seen through low and high
slots.

The profile $M_r$ is the carrier. It is the one we can place everywhere at the
next level, because its $x$-load is at most $1/2$ in both classes. The profile
$C_r$ is the stable source of weight described above: it has value $5/8$, puts
weight on the low $x$-side, and leaves the high $x$-side empty. The profile
$R_r$ softens a carrier by lowering the high $x$-load, and $D_r$ carries extra
value while keeping enough control in the joint loads. The recursion moves
weight through these three supporting shapes until it returns to a larger
value for the $M$-profile.

\begin{equation}\label{eq:recursive-profile-caps}
\begin{array}{c|c|c|c}
\text{profile}&X&Y&Z\\ \hline
M&
(1/2,1/2)&(1,1)&
\begin{pmatrix}1/2&1/2\\1/2&1/2\end{pmatrix}\\[2mm]
R&
(1/2,1/4)&(1/2,1/2)&
\begin{pmatrix}1/2&1/2\\1/4&1/4\end{pmatrix}\\[2mm]
C&
(1,0)&(1/2,1/2)&
\begin{pmatrix}1/2&1/2\\0&0\end{pmatrix}\\[2mm]
D&
(1,3/4)&(1/2,1/2)&
\begin{pmatrix}1/2&1/2\\1/4&1/4\end{pmatrix}.
\end{array}
\end{equation}
The entries in the $X$- and $Y$-columns are ordered as
$(\mathsf L,\mathsf H)$.

\begin{lemma}\label{lem:source-profile}
There is a profile $C_1$ on $P_1$ with normalized value $5/8$ and with the
$C$-bounds in \eqref{eq:recursive-profile-caps}. The zero vector on $P_1$
serves as $M_1,R_1,D_1$.
\end{lemma}

\begin{proof}
In each of the four vertical copies of $P_0$ inside $P_1$, give weight $1/2$
to the five-cycle vertices $\mathcal C=\{C_0,C_1,C_2,C_3,Q\}$, and give weight
zero to every other rectangle. The total value is
\[ 4\cdot 5\cdot\frac12=10=\frac58q_1. \]
This weighting is feasible: positive weight appears only inside vertical
copies of the five-cycle, and every clique in a five-cycle has at most two
vertices.

Only the old $Y$-slots $0$ and $1$ occur among these weighted vertices, so the
high $x$-class has load zero. In a low $x$-slot, the possible weighted
vertices in one vertical copy form either the path $C_0-Q-C_3$ or the path
$C_1-C_2-C_3$; every clique in either path has weight at most one. Thus
$X_{C_1}\le(1,0)$.

Fixing a $y$-slot in a vertical copy fixes an old $X$-slot of the base
configuration. By \Cref{lem:base-x}, the base vertices using one old $X$-slot
are independent, so such a clique contains at most one weighted vertex. This
gives $Y_{C_1}\le(1/2,1/2)$. If both an $x$-slot and a $y$-slot are fixed, the
same argument gives at most one weighted vertex; in the high $x$-class there
is no weighted vertex at all. This gives the displayed $Z$-bounds for $C$.
The zero vector plainly satisfies the remaining profile bounds.
\end{proof}

With the four profiles in hand, the recursive move is easy to state before
checking. The source $C_r$ is preserved. A diluted copy of $M_r$, paired with a
half-weight transposed copy of $C_r$, makes the more balanced helper
$R_{r+1}$. A half-weight $C_r$ and the old $D_r$ make the next $D$-profile.
Finally, $R_r,C_r,D_r$ are assembled into the next carrier $M_{r+1}$. The
figure records the roles of the four moves; the precise low/high placement in
horizontal and vertical copies is given immediately afterward.

\begin{figure}[H]
\centering
\begin{tikzpicture}[
font=\small,
input/.style={draw,rounded corners=2pt,fill=gray!6,inner xsep=5pt,inner ysep=3pt,minimum width=2.6cm},
output/.style={draw,rounded corners=2pt,fill=blue!5,inner xsep=5pt,inner ysep=3pt,minimum width=1.5cm},
movelabel/.style={font=\scriptsize,fill=white,inner sep=1pt},
arrow/.style={-{Latex[length=2mm]},line width=.45pt}
]
\node[input] (i1) at (0,0) {$C_r$};
\node[output] (o1) at (6.2,0) {$C_{r+1}$};
\draw[arrow] (i1) -- node[movelabel,above]{preserve source} (o1);

\node[input] (i2) at (0,-1.05) {$\frac12M_r,\ \frac12C_r^T$};
\node[output] (o2) at (6.2,-1.05) {$R_{r+1}$};
\draw[arrow] (i2) -- node[movelabel,above]{soften $M$} (o2);

\node[input] (i3) at (0,-2.1) {$\frac12C_r,\ D_r^T$};
\node[output] (o3) at (6.2,-2.1) {$D_{r+1}$};
\draw[arrow] (i3) -- node[movelabel,above]{carry value} (o3);

\node[input] (i4) at (0,-3.15) {$R_r,\ C_r,\ D_r$};
\node[output] (o4) at (6.2,-3.15) {$M_{r+1}$};
\draw[arrow] (i4) -- node[movelabel,above]{build carrier} (o4);
\end{tikzpicture}
\caption{The four moves in the recursive fractional construction. The diagram
shows which old profiles feed each new profile; the exact low/high
horizontal-vertical placement is given in \eqref{eq:recursive-substitutions}.}
\label{fig:recursive-profile-flow}
\end{figure}
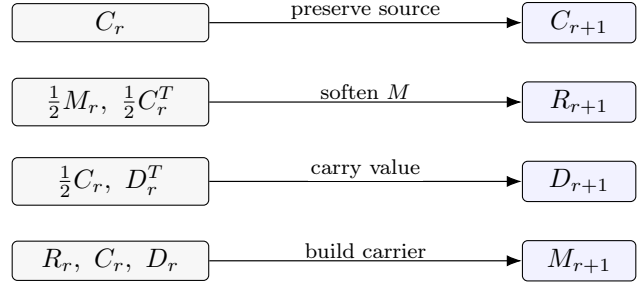

\begin{lemma}\label{lem:recursive-profile-step}
Assume that, on $P_r$ with $r\ge 1$, there are profiles
$M_r,R_r,C_r,D_r$ satisfying \eqref{eq:recursive-profile-caps}, with
normalized values $m_r,r_r,c,d_r$, where $c=5/8$. Then there are profiles
$M_{r+1},R_{r+1},C_{r+1},D_{r+1}$ satisfying the same bounds on $P_{r+1}$,
with $c$ unchanged and
\[
r_{r+1}=\frac{m_r+c}{2},\qquad
d_{r+1}=\frac{d_r}{2}+\frac{3c}{4},\qquad
m_{r+1}=\frac{r_r+c+d_r}{2}.
\]
\end{lemma}

\begin{proof}
The four profiles at level $r+1$ are obtained by placing level-$r$ profiles
in the low and high classes of horizontal and vertical copies as follows:
\begin{equation}\label{eq:recursive-substitutions}
\begin{array}{c|c|c|c}
\text{output}&(h_{\mathsf L},h_{\mathsf H})&
(v_{\mathsf L},v_{\mathsf H})&\text{normalized value}\\ \hline
C_{r+1}&(0,0)&(C_r^T,C_r^T)&c\\
R_{r+1}&(\frac12M_r,\frac12M_r)&
(\frac12C_r^T,\frac12C_r^T)&(m_r+c)/2\\
D_{r+1}&(\frac12C_r,\frac12C_r)&
(\frac12C_r^T,D_r^T)&d_r/2+3c/4\\
M_{r+1}&(0,R_r)&(C_r,D_r)&(r_r+c+d_r)/2.
\end{array}
\end{equation}
The value column follows because each of the two slot classes contains
$q_r/2$ copy indices in each direction.

It remains to check feasibility and the new load bounds. Applying
\Cref{lem:two-class-loads} to the bounds in
\eqref{eq:recursive-profile-caps} gives the following table. In the mixed-load
column, rows are horizontal-copy classes and columns are vertical-copy
classes, both ordered as $\mathsf L,\mathsf H$.

For instance, the row producing $M_{r+1}$ is the payoff. Low horizontal copies
are left empty, high horizontal copies receive $R_r$, low vertical copies
receive $C_r$, and high vertical copies receive $D_r$. The mixed loads are
\[
\begin{pmatrix}
0+1&0+1\\
1/2+0&1/4+3/4
\end{pmatrix},
\]
so no mixed clique exceeds one. The tight high-horizontal/high-vertical entry
is
\[
X_R(\mathsf H)+X_D(\mathsf H)=1/4+3/4=1.
\]
The new $x$-loads ask for a different check. The vertical-only contribution is
at most $1/2$. If the clique also has a horizontal part, then the
high-horizontal/high-vertical case uses the joint-load bound of $D_r$:
\[
X_R(\mathsf H)+Z_D(\mathsf H,f)\le 1/4+1/4=1/2
\]
for $f\in\{\mathsf L,\mathsf H\}$. The other high-horizontal case, with a low
vertical copy, uses the empty high $x$-side of $C_r$:
\[
X_R(\mathsf L)+Z_C(\mathsf H,f)\le 1/2+0=1/2.
\]
Low horizontal copies contribute nothing. Hence every new $x$-load is at most
$1/2$, which is the half-load property needed for the next round. The table
records the same check for all four outputs.

\begin{center}
\small
\setlength{\arraycolsep}{5pt}
\renewcommand{\arraystretch}{1.45}
\begin{tabular}{c|c|c|c|c}
output&mixed loads&$X'$&$Y'$&$Z'$\\ \hline
$C$&
$\begin{pmatrix}1/2&1/2\\1/2&1/2\end{pmatrix}$&
$(1,0)$&$(1/2,1/2)$&
$\begin{pmatrix}1/2&1/2\\0&0\end{pmatrix}$\\[2mm]
$R$&
$\begin{pmatrix}1/2&1/2\\1/2&1/2\end{pmatrix}$&
$(1/2,1/4)$&$(1/2,1/2)$&
$\begin{pmatrix}1/2&1/2\\1/4&1/4\end{pmatrix}$\\[2mm]
$D$&
$\begin{pmatrix}3/4&1/2\\3/4&1/2\end{pmatrix}$&
$(1,3/4)$&$(1/2,1/2)$&
$\begin{pmatrix}1/2&1/2\\1/4&1/4\end{pmatrix}$\\[2mm]
$M$&
$\begin{pmatrix}1&1\\1/2&1\end{pmatrix}$&
$(1/2,1/2)$&$(1,1)$&
$\begin{pmatrix}1/2&1/2\\1/2&1/2\end{pmatrix}$.
\end{tabular}
\end{center}

Every entry in the mixed-load column is at most one, so all four output
profiles are feasible. The remaining columns are exactly the required profile
bounds. This proves the recursive step.
\end{proof}

At this point the geometric part of the recursion is complete: every profile
needed at level $r+1$ has been made from profiles at level $r$. What remains
is to follow the total values around this loop.

\begin{proposition}\label{prop:recursive-profile-values}
For every $r\ge 1$, there is a feasible profile $M_r$ on $P_r$ with
$X_{M_r}(s)\le 1/2$ for every $x$-slot $s$ and with normalized value
\[
m_r=\frac54-\frac{30r+55+5(-1)^r}{2^{r+5}}.
\]
\end{proposition}

\begin{proof}
Begin with the profiles from \Cref{lem:source-profile} and apply
\Cref{lem:recursive-profile-step} repeatedly. The values satisfy
\[
r_{r+1}=\frac{m_r+5/8}{2},\qquad
d_{r+1}=\frac{d_r}{2}+\frac{15}{32},\qquad
m_{r+1}=\frac{r_r+5/8+d_r}{2},
\]
with $m_1=r_1=d_1=0$. The second recurrence gives
\[ d_r=\frac{15}{16}\left(1-\frac{1}{2^{r-1}}\right). \]
Combining this with $r_r=(m_{r-1}+5/8)/2$ gives the two-step recurrence\footnote{The two-step form has a simple meaning in the profile loop. The carrier
profile $M$ does not make the next carrier directly: it first passes through
the helper profile $R$, and that $R$ is used one round later to build a new
$M$. Thus, after $d_r$ has been written explicitly, the values of $M$ advance
by two levels at a time. The even and odd levels approach the same limit along
slightly different tracks; this is the source of the alternating term in the
displayed formula.}
\[
m_{r+1}=\frac14m_{r-1}+\frac{15}{16}-\frac{15}{2^{r+4}}.
\]
It remains only to check the closed form. The displayed
formula for $m_r$ holds at $r=1$ and $r=2$, and substituting it into this
two-step recurrence gives the same formula at level $r+1$. The $X$-load bound
is part of the $M$-profile bounds.
\end{proof}

Finally, the point of the $M$-profile is that it can be used uniformly in the
next crossbar. This also gives us the lower bound.

\begin{proposition}\label{prop:recursive-lp-lower}
For every $r\ge 1$,
\[
\frac{\alpha^*(P_{r+1})}{\nu(P_{r+1})}
\ge
\frac52-\frac{30r+55+5(-1)^r}{2^{r+4}}.
\]
Consequently,
\[ \liminf_{r\to\infty}\frac{\alpha^*(P_r)}{\nu(P_r)}\ge\frac52. \]
\end{proposition}

\begin{proof}
Place a copy of the profile $M_r$ from \Cref{prop:recursive-profile-values} in
every horizontal and every vertical copy of $P_r$ inside $P_{r+1}=C(P_r)$.
Local cliques are feasible. A mixed clique meets one horizontal copy and one
vertical copy, and each side contributes at most $1/2$ to the relevant old
$x$-slot class. Hence every mixed clique has weight at most one.

There are $2q_r$ copies, each carrying total weight $m_rq_r$. Since
$q_{r+1}=q_r^2$, the total weight is $2m_rq_{r+1}$. By
\Cref{prop:independence-recursion}, $\nu(P_{r+1})=q_{r+1}$, so the ratio is at
least $2m_r$, which is the displayed bound.
\end{proof}

Since $\alpha^*(P_r)\le\tau(P_r)$ for rectangle families, this also gives
\[ \liminf_{r\to\infty}\frac{\tau(P_r)}{\nu(P_r)}\ge\frac52. \]
This proves the lower-bound half of \Cref{thm:endpoint}. The next section gives
the matching piercing sets, showing that this recursive family cannot exceed
$5/2$.

\section{The matching upper bound}\label{sec:upper-bound}

The previous section built large fractional packings. We now show that the
same recursive family has piercing sets of the matching asymptotic size. The
upper bound comes from using the same lines that were built into the
construction: all chosen points will be intersections of the vertical and
horizontal marker lines.

For a rectangle $R\in P_r$, write $X_r(R)$ and $Y_r(R)$ for the vertical and
horizontal marker slots contained in $R$. The marker-grid point $(s,t)$ is the
intersection of vertical marker $s$ and horizontal marker $t$, and it lies in
$R$ exactly when
\[ s\in X_r(R)\quad\text{and}\quad t\in Y_r(R). \]
Thus a set $S\subseteq[q_r]\times[q_r]$ gives a piercing set for $P_r$ if
\[ S\cap\bigl(X_r(R)\times Y_r(R)\bigr)\ne\emptyset \]
for every $R\in P_r$.

A well-chosen grid set can pierce one level. To pass to the next level, we
keep one additional piece of information: how many chosen points lie on each marker.
If there are $q$ vertical markers and the target size is $(5/2)q$, then the
average number of chosen points on a vertical marker is $5/2$, and the same
average holds for the horizontal markers. Since these counts are integral, the
natural shape is that every marker carries either two or three chosen points,
with the extra point assigned to exactly half of the markers in each
direction. For
$S\subseteq[q]\times[q]$, define
\[
d_S^x(s)=|\{t:(s,t)\in S\}|,\qquad
d_S^y(t)=|\{s:(s,t)\in S\}|.
\]
These are the numbers of chosen grid points on a vertical marker and on a
horizontal marker.

At the next crossbar step, it matters which markers carry the third point. A
horizontal copy and a vertical copy must contribute the
same number of local choices inside their common cell, so that the choices can
be paired. We therefore keep two patterns, allowing the third points to be
placed on one half of the markers or on the other half.

The base case is small enough to see directly. The two patterns below each
have ten chosen entries. Rows are indexed by the vertical slot and columns by
the horizontal slot:
\[
M^0=
\begin{pmatrix}
1&0&1&1\\
0&1&1&0\\
1&1&1&0\\
1&0&0&1
\end{pmatrix},
\qquad
M^1=
\begin{pmatrix}
1&0&0&1\\
1&1&1&0\\
0&1&0&1\\
0&1&1&1
\end{pmatrix}.
\]
Let $S_0^\delta$ be the set of entries equal to one in $M^\delta$.

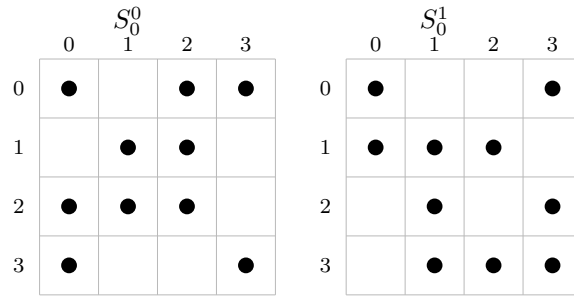
\begin{figure}[H]
\centering
\begin{tikzpicture}[
scale=.78,
gridline/.style={gray!55,line width=.25pt},
dot/.style={circle,fill=black,inner sep=2.1pt},
lab/.style={font=\scriptsize}
]
\begin{scope}
\node[font=\small] at (1.5,4.65) {$S_0^0$};
\foreach \x in {0,1,2,3,4} {
  \draw[gridline] (\x,0) -- (\x,4);
  \draw[gridline] (0,\x) -- (4,\x);
}
\foreach \t in {0,1,2,3} \node[lab] at (\t+.5,4.25) {$\t$};
\foreach \s/\y in {0/3.5,1/2.5,2/1.5,3/.5}
  \node[lab] at (-.35,\y) {$\s$};
\foreach \s/\t in {0/0,0/2,0/3,1/1,1/2,2/0,2/1,2/2,3/0,3/3}
  \node[dot] at (\t+.5,3.5-\s) {};
\end{scope}

\begin{scope}[xshift=5.2cm]
\node[font=\small] at (1.5,4.65) {$S_0^1$};
\foreach \x in {0,1,2,3,4} {
  \draw[gridline] (\x,0) -- (\x,4);
  \draw[gridline] (0,\x) -- (4,\x);
}
\foreach \t in {0,1,2,3} \node[lab] at (\t+.5,4.25) {$\t$};
\foreach \s/\y in {0/3.5,1/2.5,2/1.5,3/.5}
  \node[lab] at (-.35,\y) {$\s$};
\foreach \s/\t in {0/0,0/3,1/0,1/1,1/2,2/1,2/3,3/1,3/2,3/3}
  \node[dot] at (\t+.5,3.5-\s) {};
\end{scope}
\end{tikzpicture}
\caption{The two base marker-grid patterns. A dot in row $s$ and column $t$
means that the marker-grid point $(s,t)$ is chosen.}
\label{fig:base-piercing-patterns}
\end{figure}

\Cref{fig:base-piercing-geometry-both} places both patterns directly on the
base drawing. The same dots appear in \Cref{fig:base-piercing-patterns}; the
geometric picture shows them as points in the plane, while the grid picture
records the row and column counts. Each panel should be read separately:
either one of these two dot sets is already a piercing set.

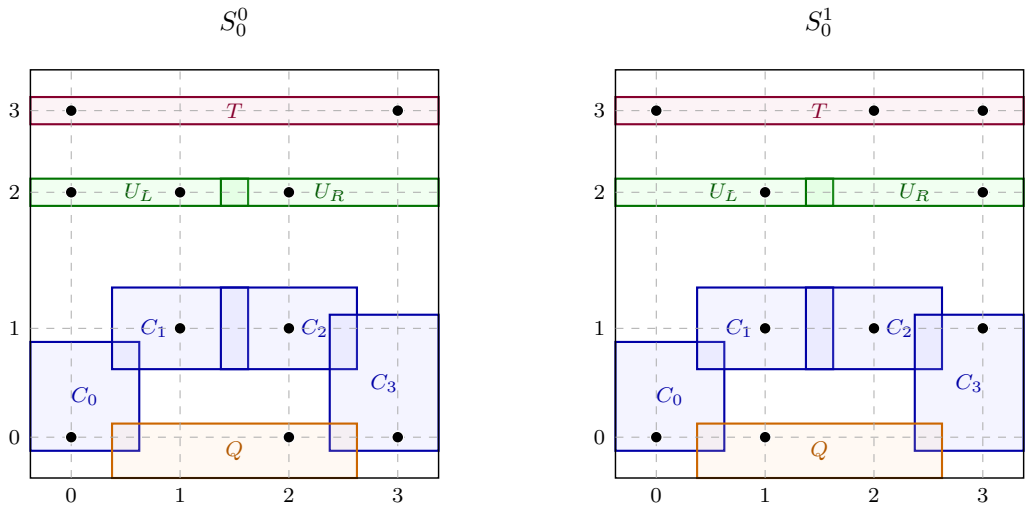
\begin{figure}[H]
\centering
\begin{tikzpicture}[
scale=0.18,
rect/.style={line width=.8pt,fill opacity=.16},
marker/.style={gray!65,dashed,line width=.35pt},
pt/.style={circle,fill=black,draw=white,line width=.45pt,inner sep=1.55pt},
lab/.style={font=\scriptsize}
]
\begin{scope}
\node[font=\small] at (15,33.5) {$S_0^0$};
\filldraw[rect,fill=purple!30,draw=purple!70!black] (0,26) rectangle (30,28);
\node[lab,text=purple!70!black] at (15,27) {$T$};

\filldraw[rect,fill=green!35,draw=green!45!black] (0,20) rectangle (16,22);
\node[lab,text=green!35!black] at (8,21) {$U_L$};
\filldraw[rect,fill=green!35,draw=green!45!black] (14,20) rectangle (30,22);
\node[lab,text=green!35!black] at (22,21) {$U_R$};

\filldraw[rect,fill=blue!25,draw=blue!65!black] (0,2) rectangle (8,10);
\node[lab,text=blue!65!black] at (4,6) {$C_0$};
\filldraw[rect,fill=blue!25,draw=blue!65!black] (6,8) rectangle (16,14);
\node[lab,text=blue!65!black] at (9.1,11) {$C_1$};
\filldraw[rect,fill=blue!25,draw=blue!65!black] (14,8) rectangle (24,14);
\node[lab,text=blue!65!black] at (20.9,11) {$C_2$};
\filldraw[rect,fill=blue!25,draw=blue!65!black] (22,2) rectangle (30,12);
\node[lab,text=blue!65!black] at (26,7) {$C_3$};
\filldraw[rect,fill=orange!30,draw=orange!80!black] (6,0) rectangle (24,4);
\node[lab,text=orange!70!black] at (15,2) {$Q$};

\draw[black,line width=.55pt] (0,0) rectangle (30,30);
\foreach \x/\labx in {3/0,11/1,19/2,27/3} {
  \draw[marker] (\x,0) -- (\x,30);
  \node[lab,below] at (\x,0) {$\labx$};
}
\foreach \y/\laby in {3/0,11/1,21/2,27/3} {
  \draw[marker] (0,\y) -- (30,\y);
  \node[lab,left] at (0,\y) {$\laby$};
}
\foreach \x/\y in {
  3/3,3/21,3/27,
  11/11,11/21,
  19/3,19/11,19/21,
  27/3,27/27
} {
  \node[pt] at (\x,\y) {};
}
\end{scope}

\begin{scope}[xshift=43cm]
\node[font=\small] at (15,33.5) {$S_0^1$};
\filldraw[rect,fill=purple!30,draw=purple!70!black] (0,26) rectangle (30,28);
\node[lab,text=purple!70!black] at (15,27) {$T$};

\filldraw[rect,fill=green!35,draw=green!45!black] (0,20) rectangle (16,22);
\node[lab,text=green!35!black] at (8,21) {$U_L$};
\filldraw[rect,fill=green!35,draw=green!45!black] (14,20) rectangle (30,22);
\node[lab,text=green!35!black] at (22,21) {$U_R$};

\filldraw[rect,fill=blue!25,draw=blue!65!black] (0,2) rectangle (8,10);
\node[lab,text=blue!65!black] at (4,6) {$C_0$};
\filldraw[rect,fill=blue!25,draw=blue!65!black] (6,8) rectangle (16,14);
\node[lab,text=blue!65!black] at (9.1,11) {$C_1$};
\filldraw[rect,fill=blue!25,draw=blue!65!black] (14,8) rectangle (24,14);
\node[lab,text=blue!65!black] at (20.9,11) {$C_2$};
\filldraw[rect,fill=blue!25,draw=blue!65!black] (22,2) rectangle (30,12);
\node[lab,text=blue!65!black] at (26,7) {$C_3$};
\filldraw[rect,fill=orange!30,draw=orange!80!black] (6,0) rectangle (24,4);
\node[lab,text=orange!70!black] at (15,2) {$Q$};

\draw[black,line width=.55pt] (0,0) rectangle (30,30);
\foreach \x/\labx in {3/0,11/1,19/2,27/3} {
  \draw[marker] (\x,0) -- (\x,30);
  \node[lab,below] at (\x,0) {$\labx$};
}
\foreach \y/\laby in {3/0,11/1,21/2,27/3} {
  \draw[marker] (0,\y) -- (30,\y);
  \node[lab,left] at (0,\y) {$\laby$};
}
\foreach \x/\y in {
  3/3,3/27,
  11/3,11/11,11/21,
  19/11,19/27,
  27/11,27/21,27/27
} {
  \node[pt] at (\x,\y) {};
}
\end{scope}
\end{tikzpicture}
\caption{Both base marker-grid patterns placed on the base drawing, shown in
separate panels. Each panel is a piercing set for the eight base rectangles.}
\label{fig:base-piercing-geometry-both}
\end{figure}

The patterns must also meet the rectangles while keeping the right row and
column counts. The full check is short; one possible choice of a point for
each rectangle is
\[
\begin{array}{c|c|c}
R&\text{point in }S_0^0&\text{point in }S_0^1\\ \hline
T&(0,3)&(0,3)\\
U_L&(0,2)&(1,2)\\
U_R&(2,2)&(3,2)\\
C_0&(0,0)&(0,0)\\
C_1&(1,1)&(1,1)\\
C_2&(2,1)&(2,1)\\
C_3&(3,0)&(3,1)\\
Q&(2,0)&(1,0).
\end{array}
\]
Each displayed point lies in $X_0(R)\times Y_0(R)$. If
\[ \varepsilon_0(s)=s\pmod 2, \]
then the row and column counts of the two patterns satisfy
\[
d_{S_0^\delta}^x(s)=d_{S_0^\delta}^y(s)
=2+\mathbf 1_{\varepsilon_0(s)=\delta}.
\]
Thus every row and every column contains either two or three chosen points,
and exactly half of the slots receive the third point.

The recursive statement is the following.

\begin{lemma}\label{lem:recursive-piercing-patterns}
For every $r\ge 0$, there is a balanced map
\[ \varepsilon_r:[q_r]\to\{0,1\}, \]
where balanced means that each value occurs $q_r/2$ times. There are also two
marker-grid piercing sets
$S_r^0,S_r^1\subseteq[q_r]\times[q_r]$ such that, for every
$s\in[q_r]$ and $\delta\in\{0,1\}$,
\[
d_{S_r^\delta}^x(s)=d_{S_r^\delta}^y(s)
=2+\mathbf 1_{\varepsilon_r(s)=\delta}.
\]
In particular, $|S_r^\delta|=(5/2)q_r$.
\end{lemma}

The proof idea is to build the next pattern cell by cell. Consider one cell in
the next crossbar arrangement, where a horizontal copy $H_i$ meets a vertical
copy $V_j$. A point chosen in the horizontal copy has local form $(j,b)$: it
uses outer column $j$ and an old horizontal slot $b$. A point chosen in the
vertical copy has local form $(i,t)$: it uses outer row $i$ and an old
horizontal slot $t$ after the coordinate swap. A new marker-grid point inside
that cell records both local choices by
\[ (j,b),\ (i,t)\quad\longmapsto\quad ((j,t),(i,b)). \]
The figure records this pairing step. The proof will choose the two old
patterns in each copy so that the available $(j,b)$ choices and $(i,t)$
choices come in equal numbers for every pair $(i,j)$; once that happens, any
bijection gives the needed marker-grid points.

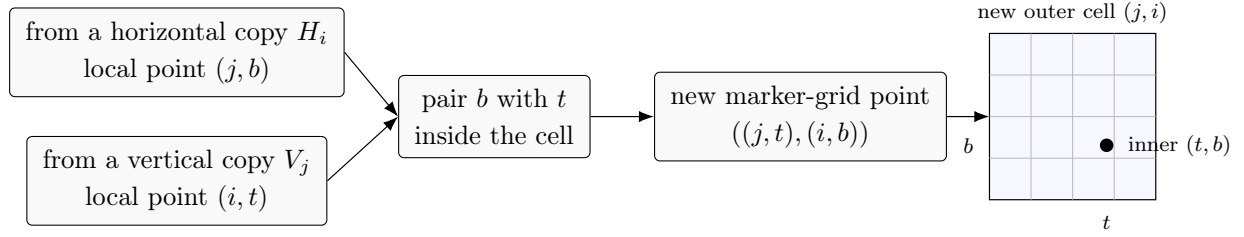
\begin{figure}[H]
\centering
\begin{tikzpicture}[
font=\small,
arrow/.style={-{Latex[length=2mm]},line width=.45pt},
box/.style={draw,rounded corners=2pt,fill=gray!5,inner xsep=6pt,inner ysep=5pt},
cell/.style={draw,line width=.45pt,fill=blue!3},
dot/.style={circle,fill=black,inner sep=1.8pt},
tinylabel/.style={font=\scriptsize}
]
\node[box,align=center] (h) at (0,0.85) {
from a horizontal copy $H_i$\\[2pt]
local point $(j,b)$
};
\node[box,align=center] (v) at (0,-0.85) {
from a vertical copy $V_j$\\[2pt]
local point $(i,t)$
};
\node[box,align=center] (pair) at (4.2,0) {
pair $b$ with $t$\\[2pt]
inside the cell
};
\node[box,align=center] (new) at (8.25,0) {
new marker-grid point\\[2pt]
$((j,t),(i,b))$
};

\begin{scope}[xshift=10.75cm,yshift=-1.1cm]
\draw[cell] (0,0) rectangle (2.2,2.2);
\draw[gray!55,line width=.25pt] (.55,0) -- (.55,2.2);
\draw[gray!55,line width=.25pt] (1.1,0) -- (1.1,2.2);
\draw[gray!55,line width=.25pt] (1.65,0) -- (1.65,2.2);
\draw[gray!55,line width=.25pt] (0,.55) -- (2.2,.55);
\draw[gray!55,line width=.25pt] (0,1.1) -- (2.2,1.1);
\draw[gray!55,line width=.25pt] (0,1.65) -- (2.2,1.65);
\coordinate (gridwest) at (0,1.1);
\node[tinylabel] at (1.1,2.48) {new outer cell $(j,i)$};
\node[dot] at (1.55,.72) {};
\node[tinylabel,anchor=west] at (1.7,.72) {inner $(t,b)$};
\node[tinylabel,anchor=north] at (1.55,-.08) {$t$};
\node[tinylabel,anchor=east] at (-.08,.72) {$b$};
\end{scope}

\draw[arrow] (h.east) -- (pair.west);
\draw[arrow] (v.east) -- (pair.west);
\draw[arrow] (pair.east) -- (new.west);
\draw[arrow] (new.east) -- (gridwest);
\end{tikzpicture}
\caption{How one paired local choice becomes a marker-grid point at the next
level. In the new grid, the outer coordinates are $(j,i)$ and the inner
coordinates are $(t,b)$.}
\label{fig:piercing-pairing}
\end{figure}

The first repeated family shows the rule in action. The pairing rule produces
forty marker-grid points, shown in \Cref{fig:p1-piercing-overlay} on the
fatter realization from \Cref{fig:p1-fat}. The inner coordinates in this
picture are old horizontal slots, whose marker positions are
$\eta=(1.4,3.6,7.5,9.5)$. Thus the new $x$-slot $(j,t)$ is drawn at
$30j+5+\eta_t$, and the new $y$-slot $(i,b)$ is drawn at $30i+5+\eta_b$.
This is the coordinate rule from \Cref{fig:piercing-pairing}, now placed
directly on the rectangles.

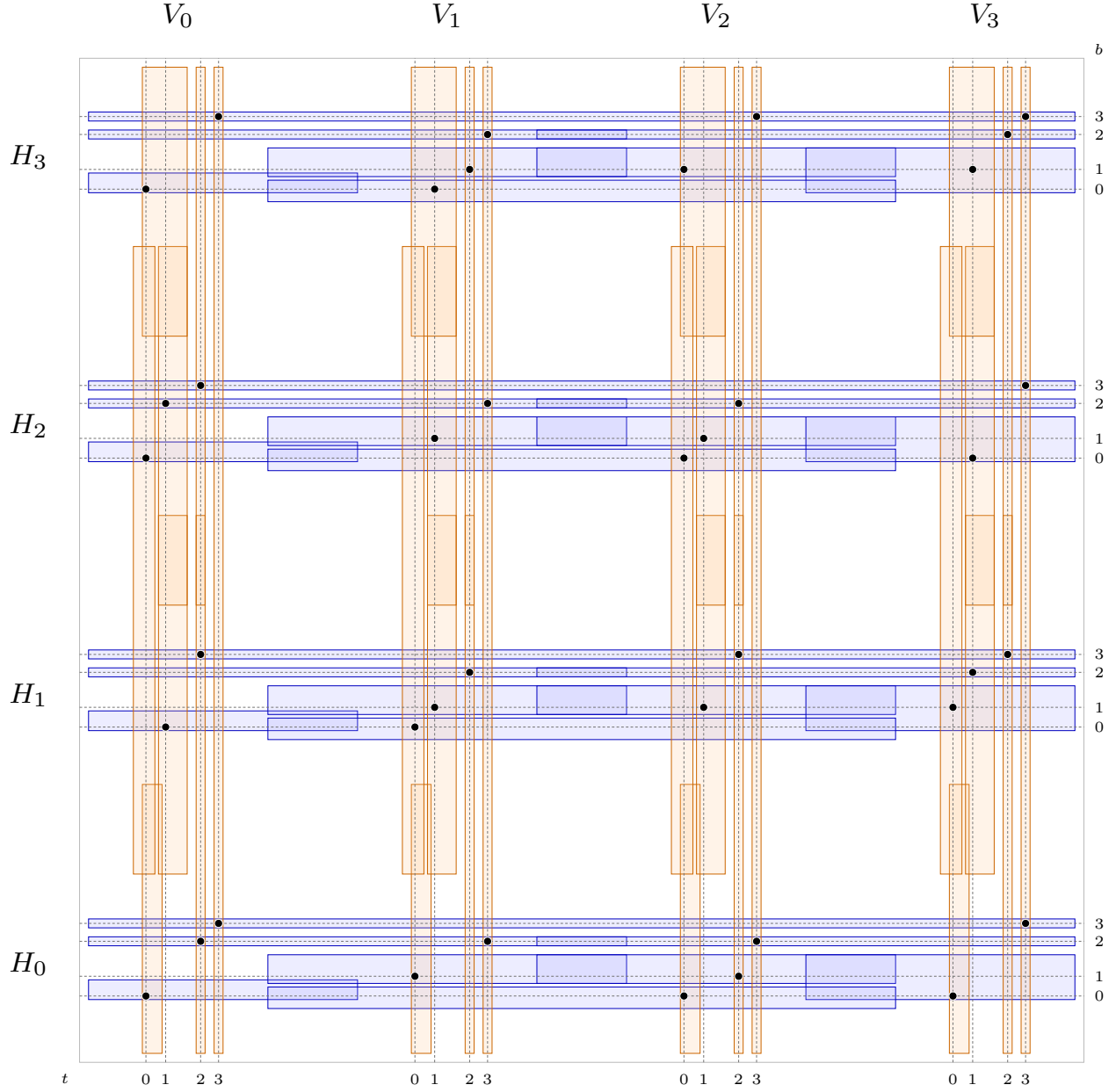
\begin{figure}[H]
\centering
\begin{tikzpicture}[
scale=1.5, transform shape, 
x=.09cm,y=.09cm,
hrect/.style={fill=blue!35,draw=blue!75!black,line width=.24pt,fill opacity=.20},
vrect/.style={fill=orange!45,draw=orange!80!black,line width=.24pt,fill opacity=.20},
marker/.style={black!55,dash pattern=on 1.1pt off .9pt,line width=.34pt},
pierce/.style={circle,fill=black,draw=white,line width=.15pt,inner sep=.82pt},
copylabel/.style={font=\scriptsize},
slotlabel/.style={font=\tiny,scale=.72,transform shape}
]
\draw[gray!55,line width=.25pt] (-1,-1) rectangle (111,111);

\foreach \i in {0,1,2,3} {
  \pgfmathsetmacro{\yc}{30*\i+10}
  \node[copylabel,anchor=east] at (-3,\yc) {$H_{\i}$};
  \foreach \xa/\xb/\ya/\yb in {
    0/11/9/10,
    0/6/7/8,
    5/11/7/8,
    0/3/1/3.2,
    2/6/2.8/6,
    5/9/2.8/6,
    8/11/1/6,
    2/9/0/2.4
  } {
    \pgfmathsetmacro{\xone}{10*\xa}
    \pgfmathsetmacro{\xtwo}{10*\xb}
    \pgfmathsetmacro{\yone}{30*\i+5+\ya}
    \pgfmathsetmacro{\ytwo}{30*\i+5+\yb}
    \filldraw[hrect] (\xone,\yone) rectangle (\xtwo,\ytwo);
  }
}

\foreach \j in {0,1,2,3} {
  \pgfmathsetmacro{\xc}{30*\j+10}
  \node[copylabel,anchor=south] at (\xc,113) {$V_{\j}$};
  \foreach \xa/\xb/\ya/\yb in {
    0/11/9/10,
    0/6/7/8,
    5/11/7/8,
    0/3/1/3.2,
    2/6/2.8/6,
    5/9/2.8/6,
    8/11/1/6,
    2/9/0/2.4
  } {
    \pgfmathsetmacro{\xone}{30*\j+5+\ya}
    \pgfmathsetmacro{\xtwo}{30*\j+5+\yb}
    \pgfmathsetmacro{\yone}{10*\xa}
    \pgfmathsetmacro{\ytwo}{10*\xb}
    \filldraw[vrect] (\xone,\yone) rectangle (\xtwo,\ytwo);
  }
}

\foreach \j in {0,1,2,3} {
  \foreach \t/\pos in {0/1.4,1/3.6,2/7.5,3/9.5} {
    \pgfmathsetmacro{\mx}{30*\j+5+\pos}
    \draw[marker] (\mx,-1) -- (\mx,111);
    \node[slotlabel,anchor=north] at (\mx,-1.25) {$\t$};
  }
}
\foreach \i in {0,1,2,3} {
  \foreach \b/\pos in {0/1.4,1/3.6,2/7.5,3/9.5} {
    \pgfmathsetmacro{\my}{30*\i+5+\pos}
    \draw[marker] (-1,\my) -- (111,\my);
    \node[slotlabel,anchor=west] at (111.15,\my) {$\b$};
  }
}
\node[slotlabel,anchor=north east] at (-1.15,-1.25) {$t$};
\node[slotlabel,anchor=west] at (111.15,112) {$b$};

\foreach \j/\t/\i/\b in {
  0/0/0/0,0/2/0/2,0/3/0/3,1/0/0/1,1/3/0/2,
  2/0/0/0,2/2/0/1,2/3/0/2,3/0/0/0,3/3/0/3,
  0/1/1/0,0/2/1/3,1/0/1/0,1/1/1/1,1/2/1/2,
  2/1/1/1,2/2/1/3,3/0/1/1,3/1/1/2,3/2/1/3,
  0/0/2/0,0/1/2/2,0/2/2/3,1/1/2/1,1/3/2/2,
  2/0/2/0,2/1/2/1,2/2/2/2,3/1/2/0,3/3/2/3,
  0/0/3/0,0/3/3/3,1/1/3/0,1/2/3/1,1/3/3/2,
  2/0/3/1,2/3/3/3,3/1/3/1,3/2/3/2,3/3/3/3
} {
  \pgfmathsetmacro{\xinner}{(\t==0 ? 1.4 : (\t==1 ? 3.6 : (\t==2 ? 7.5 : 9.5)))}
  \pgfmathsetmacro{\yinner}{(\b==0 ? 1.4 : (\b==1 ? 3.6 : (\b==2 ? 7.5 : 9.5)))}
  \pgfmathsetmacro{\px}{30*\j+5+\xinner}
  \pgfmathsetmacro{\py}{30*\i+5+\yinner}
  \node[pierce] at (\px,\py) {};
}
\end{tikzpicture}
\caption{The first recursive piercing pattern $S_1^0$ overlaid on the
fatter 64-rectangle drawing. The black dots are the forty marker-grid points
produced by the pairing rule; the dashed lines show the marker-grid lines
$x=30j+5+\eta_t$ and $y=30i+5+\eta_b$, labelled by $t$ below the drawing and
by $b$ on the right.}
\label{fig:p1-piercing-overlay}
\end{figure}

The picture suggests the induction. At each level there will again be two
piercing patterns. Numerically, what has to be preserved is the same
two-or-three count on every vertical and horizontal marker.

Here and in the proof, $\oplus$ denotes addition modulo two. The pattern
choices in the proof are arranged so that, inside each cell $(i,j)$, the
horizontal and vertical sides receive their extra local choice under the same
condition, namely $\varepsilon_r(i)\oplus\varepsilon_r(j)=\delta$.

\begin{proof}
For $r=0$, take $\varepsilon_0(s)=s\pmod 2$. The witness table above shows
that both displayed patterns pierce $P_0$, and the row and column counts give
the required formula for $d^x$ and $d^y$.

Assume the statement at level $r$, and put $q=q_r$. Fix
$\delta\in\{0,1\}$. The new pattern will be assembled one pair of copies at a
time. In the horizontal copy with index $i$, use the old pattern
\[ \phi_i=\delta\oplus\varepsilon_r(i), \]
and in the vertical copy with index $j$, use the old pattern
\[ \psi_j=\delta\oplus\varepsilon_r(j). \]

Fix a horizontal index $i$ and a vertical index $j$. Let
\[
A_{i,j}=\{b:(j,b)\in S_r^{\phi_i}\},\qquad
B_{i,j}=\{t:(i,t)\in S_r^{\psi_j}\}.
\]
The row and column counts give
\[
|A_{i,j}|=2+\mathbf 1_{\varepsilon_r(j)=\phi_i}
=2+\mathbf 1_{\varepsilon_r(i)\oplus\varepsilon_r(j)=\delta},
\]
and
\[
|B_{i,j}|=2+\mathbf 1_{\varepsilon_r(i)=\psi_j}
=2+\mathbf 1_{\varepsilon_r(i)\oplus\varepsilon_r(j)=\delta}.
\]
Thus $A_{i,j}$ and $B_{i,j}$ have the same size. This equality is the point of
the choices of $\phi_i$ and $\psi_j$: the local choices available from the two
copies have been counted by the same $\oplus$-condition and can now be matched.
Choose an arbitrary bijection $\beta_{i,j}:A_{i,j}\to B_{i,j}$. For each
$b\in A_{i,j}$, put the new marker-grid point
\[ \bigl((j,\beta_{i,j}(b)),(i,b)\bigr) \]
into $S_{r+1}^\delta$. Doing this over all choices of $i$ and $j$ defines
$S_{r+1}^\delta$.

The points are distinct. For a fixed pair $(i,j)$, the values of $b$ are
distinct and the map $\beta_{i,j}$ is injective. Points from different pairs
have different outer coordinates.

Now let $H_i(R)$ be a horizontal rectangle in $P_{r+1}$. The old pattern
$S_r^{\phi_i}$ pierces $R$, so it contains a point $(j,b)$ with
$j\in X_r(R)$ and $b\in Y_r(R)$. The paired point
$((j,\beta_{i,j}(b)),(i,b))$ lies in $H_i(R)$, because
\[
X_{r+1}(H_i(R))=X_r(R)\times[q],\qquad
Y_{r+1}(H_i(R))=\{i\}\times Y_r(R).
\]
For a vertical rectangle, use the same pairing in the other direction. If
$S_r^{\psi_j}$ pierces $R$ at a point $(i,t)$, then $t\in B_{i,j}$, so there
is a $b\in A_{i,j}$ with $\beta_{i,j}(b)=t$. The resulting point
$((j,t),(i,b))$ lies in $V_j(R)$, using
\[
X_{r+1}(V_j(R))=\{j\}\times Y_r(R),\qquad
Y_{r+1}(V_j(R))=X_r(R)\times[q].
\]
Hence $S_{r+1}^\delta$ pierces every rectangle of $P_{r+1}$.

Define
\[ \varepsilon_{r+1}(j,t)=\varepsilon_r(j)\oplus\varepsilon_r(t). \]
Fix a new $x$-slot $(j,t)$. Its count in $S_{r+1}^\delta$ is the number of
indices $i$ for which $(i,t)\in S_r^{\psi_j}$. By the old count for the
horizontal marker $t$, this is
\[
2+\mathbf 1_{\varepsilon_r(t)=\psi_j}
=2+\mathbf 1_{\varepsilon_r(j)\oplus\varepsilon_r(t)=\delta}
=2+\mathbf 1_{\varepsilon_{r+1}(j,t)=\delta}.
\]
The calculation for a new $y$-slot $(i,b)$ is identical, with $i$ and $b$ in
place of $j$ and $t$.

Finally, $\varepsilon_{r+1}$ is balanced. For each fixed $j$, exactly half of
the values of $t$ make $\varepsilon_r(j)\oplus\varepsilon_r(t)$ equal to
zero, and half make it equal to one, because $\varepsilon_r$ is balanced.
This completes the induction.

The size formula follows by summing the $x$-counts:
\[
|S_r^\delta|=\sum_{s\in[q_r]}d_{S_r^\delta}^x(s)
=2q_r+\frac{q_r}{2}=\frac52q_r.
\]
\end{proof}

The same marker-grid patterns also explain why the number $5/4$ is the
ceiling for the reusable half-load profiles from the previous section. The
profiles $M_r$ were built so that every $x$-slot carries load at most $1/2$.
The piercing patterns show that this condition already imposes the right
upper bound on their possible total value.

\begin{proposition}[The half-load bound]\label{prop:half-load-bound}
Let $w$ be a nonnegative weighting on $P_r$ such that $X_w(s)\le 1/2$ for
every $x$-slot $s$ of $P_r$. Then $w(P_r)\le (5/4)q_r$.
\end{proposition}

\begin{proof}
Choose one of the marker-grid piercing sets $S_r^\delta$ from
\Cref{lem:recursive-piercing-patterns}. For each rectangle of positive
weight, choose one point of $S_r^\delta$ that pierces it, and assign the
rectangle to that point. This is possible because $S_r^\delta$ pierces
$P_r$.

Fix a marker-grid point $(s,t)\in S_r^\delta$. All rectangles assigned to
this point contain the same point of the plane, so they form a clique. They
also all use the $x$-slot $s$. By the definition of $X_w(s)$, the total
weight assigned to $(s,t)$ is at most $X_w(s)\le 1/2$. Summing over the
$|S_r^\delta|=(5/2)q_r$ chosen marker-grid points gives
$w(P_r)\le (1/2)(5/2)q_r=(5/4)q_r$.
\end{proof}

Thus the value reached in \Cref{prop:recursive-profile-values} is the largest
possible asymptotic value for any half-load weighting. When such a weighting
is placed in both the horizontal and vertical copies at the next crossbar
step, this $5/4$ ceiling becomes the $5/2$ endpoint in
\Cref{thm:endpoint}.

\begin{proposition}\label{prop:matching-upper-bound}
For every $r\ge 0$,
\[
\tau(P_r)\le \frac52 q_r.
\]
Consequently,
\[
\frac{\tau(P_r)}{\nu(P_r)}\le\frac52
\qquad\text{and}\qquad
\frac{\alpha^*(P_r)}{\nu(P_r)}\le\frac52.
\]
\end{proposition}

\begin{proof}
By \Cref{lem:recursive-piercing-patterns}, either set $S_r^\delta$ gives
$(5/2)q_r$ marker-grid points that pierce $P_r$. Thus
$\tau(P_r)\le(5/2)q_r$. By \Cref{prop:independence-recursion},
$\nu(P_r)=q_r$, giving the first ratio bound. The LP bound follows from
$\alpha^*(P_r)\le\tau(P_r)$.
\end{proof}

\begin{proof}[Completion of \Cref{thm:endpoint}]
\Cref{prop:recursive-lp-lower} gives
\[
\liminf_{r\to\infty}\frac{\alpha^*(P_r)}{\nu(P_r)}\ge\frac52.
\]
Since $\alpha^*(P_r)\le\tau(P_r)$, this also gives
\[
\liminf_{r\to\infty}\frac{\tau(P_r)}{\nu(P_r)}\ge\frac52.
\]
The upper bounds in \Cref{prop:matching-upper-bound} give the reverse
inequalities for both limsups. Therefore both limits exist and equal $5/2$.
\end{proof}

\section{Interpolation}\label{sec:interpolation}

We end with a small adjustment argument. The recursive families give ratios
arbitrarily close to $5/2$. If a desired rational value $t$ is below $5/2$,
we may first choose a family whose ratio is above $t$, and then lower the
ratio by adding isolated rectangles. An isolated rectangle contributes one to
the packing number, one to the piercing number, and one to the fractional
packing optimum. Thus it has ratio exactly one, and adding many such
rectangles lets us tune an existing ratio downward.

We shall use this only for separated disjoint unions. If two families
$\mathcal A$ and $\mathcal B$ are placed in far apart regions of the plane,
then
\[
\nu(\mathcal A\sqcup\mathcal B)=\nu(\mathcal A)+\nu(\mathcal B),
\]
\[
\tau(\mathcal A\sqcup\mathcal B)=\tau(\mathcal A)+\tau(\mathcal B),
\]
and
\[
\alpha^*(\mathcal A\sqcup\mathcal B)
=\alpha^*(\mathcal A)+\alpha^*(\mathcal B).
\]
For $\nu$, maximum disjoint subfamilies from the two sides can be combined,
and any disjoint subfamily splits into its two parts. For $\tau$, no point can
hit rectangles on both sides once the two regions are separated. For
$\alpha^*$, every clique lies entirely in one side, so the clique relaxation
splits into two independent linear programs.

\begin{proof}[Proof of \Cref{thm:interpolation}]
Let $X$ denote either $\alpha^*$ or $\tau$. We first consider a rational
number $t$ with $1<t<5/2$. By \Cref{thm:endpoint}, there is a member
$P_r$ of the recursive family such that
\[
\frac{X(P_r)}{\nu(P_r)}>t.
\]
For $X=\alpha^*$ this follows directly from the first limit in
\Cref{thm:endpoint}; for $X=\tau$ it follows from the second. Write
\[
a=\nu(P_r),\qquad b=X(P_r).
\]
Here $a$ is a positive integer. Also, $b$ is rational: for $X=\tau$ it is an
integer, and for $X=\alpha^*$ it is the optimum of a finite linear program
with rational coefficients.

Take $M$ separated copies of $P_r$ and add $N$ further isolated rectangles,
also placed far away from all the copies and from one another. By the
additivity just discussed, the resulting family has ratio
\[
\frac{Mb+N}{Ma+N}.
\]
We want this to equal $t$. Solving for $N$ gives
\[
N=M\frac{b-ta}{t-1}.
\]
Since $b/a>t$ and $t>1$, the factor $(b-ta)/(t-1)$ is positive. It is also
rational, so we choose $M$ large enough that $N$ is a positive integer. Then
the separated union $\R$ of these $M$ copies of $P_r$ and these $N$ isolated
rectangles satisfies
\[
\frac{X(\R)}{\nu(\R)}=t.
\]

Finally, the endpoint $t=1$ is obtained by any nonempty family of pairwise
disjoint isolated rectangles. Applying the argument once with $X=\alpha^*$
and once with $X=\tau$ proves both assertions.
\end{proof}

\section{Conclusion}\label{sec:conclusion}

Ajwani, Gajjala, Raman, and Ray
\cite{ajwani2026counterexamplewegnersconjectureaxisparallel} recently gave
the first counterexample to Wegner's conjecture and the finite LP benchmark
discussed above. The construction
here takes that development in a more compact and recursive direction. Its
starting point is an eight-rectangle base whose independent sets are
controlled by a few ordered slots, while the geometry remains triangle-free.
The slot lists bound independent sets, the disjointness rules give the
triangle-free 64-rectangle arrangement, and the second list system carries the
information needed for later recursive levels.

This turns the 64-rectangle counterexample into the first level of a recursive
family. The same horizontal-vertical arrangement that proves Wegner's
conjecture false also creates the regularity needed to support fractional
weightings and to construct piercing sets. At the finite level, this gives a
clique-relaxation gap larger than the Ajwani-Gajjala-Raman-Ray benchmark.
In the limit, the lower and upper arguments meet: weights are controlled by
slot loads, while piercing sets are controlled by marker-grid pairings.
Together they show that the recursive family has a precise limiting value,
\[
\lim_{r\to\infty}\frac{\alpha^*(P_r)}{\nu(P_r)}
=\lim_{r\to\infty}\frac{\tau(P_r)}{\nu(P_r)}
=\frac52.
\]
The interpolation argument then turns this limiting behavior into exact
finite realizations of every rational value below $5/2$ as each of these two
ratios.

Several natural questions remain. On the packing-piercing side, Gyárfás and
Lehel \cite{gyarfas1985coveringcoloring} asked for a constant bound on
$\tau(\R)/\nu(\R)$ for rectangle families. The known lower side has moved
from the $5/3$ barrier of Fon-Der-Flaass and Kostochka
\cite{FonDerFlaass_1993}, through the examples reported by Correa,
Feuilloley, Pérez-Lantero, and Soto
\cite{correa2015independenthitting} with $\tau(\R)=2\nu(\R)-4$, past the
first counterexample of Ajwani, Gajjala, Raman, and Ray
\cite{ajwani2026counterexamplewegnersconjectureaxisparallel}, and now to the
limit $5/2$ for the recursive family constructed here. This value is the
endpoint proved for our construction; the global supremum for all rectangle
families remains open. Known general upper bounds are still much larger, with
the current line including the $O(\nu(\R)(\log\log\nu(\R))^2)$ bound of
Correa, Feuilloley, Pérez-Lantero, and Soto
\cite{correa2015independenthitting}.

On the algorithmic side, the standard point relaxation for Maximum Independent
Set of Rectangles has long been tied to approximation algorithms, beginning
with the framework of Chalermsook and Chuzhoy
\cite{chalermsook2009maximumindependentrectangles}. Recent algorithms achieve
better constant factors, including the $(2+\epsilon)$ approximation of
Gálvez, Khan, Mari, Mömke, Reddy, and Wiese
\cite{galvez2021twoepsilonmisr}. The lower bound here shows that the standard
point, equivalently clique, relaxation cannot by itself certify approximation
factors below $5/2$. It would be interesting to determine the true integrality
gap of this relaxation, to find smaller finite examples with large gap, and to
understand which parts of the crossbar mechanism are essential.

\paragraph{AI usage acknowledgment} The authors used OpenAI’s Codex to format parts of the text and generate the figures. The final draft was reviewed by the authors.

\bibliographystyle{alpha}
\bibliography{references}

\clearpage
\appendix
\definecolor{crossTypeAA}{HTML}{8F1216}
\definecolor{crossTypeAAFill}{HTML}{FBE3E1}
\definecolor{crossTypeAB}{HTML}{0072B2}
\definecolor{crossTypeABFill}{HTML}{EAF3FA}
\definecolor{crossTypeBA}{HTML}{E69F00}
\definecolor{crossTypeBAFill}{HTML}{FFF4DD}
\definecolor{crossTypeBB}{HTML}{CC79A7}
\definecolor{crossTypeBBFill}{HTML}{F9ECF4}

\section{Graph schematics for the first crossbar}\label{app:first-crossbar-schematics}

This appendix records three schematic pictures for the first repeated family
$P_1=C(P_0)$. The first marks the crossings in the rectangle drawing, the
second enlarges one crossing of each type, and the third keeps only the graph
structure that appears inside those crossings.

\Cref{fig:app-crossings-geometry} repeats the compressed drawing of $P_1$ from
\Cref{fig:p1-rectangles}. A crossing is one of the highlighted squares where a
horizontal copy strip and a vertical copy strip meet. These sixteen crossings
are the geometric source of the $4$ by $4$ board used in the graph schematic
below.

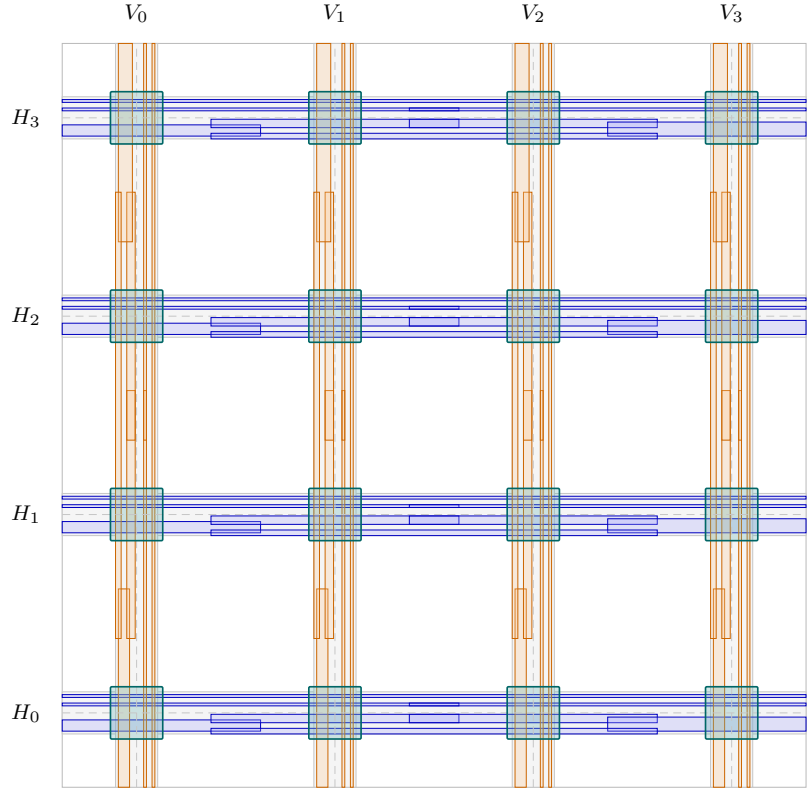
\begin{figure}[H]
\centering
\begin{tikzpicture}[
scale=.82,
hrect/.style={fill=blue!35,draw=blue!75!black,line width=.22pt,fill opacity=.28},
vrect/.style={fill=orange!42,draw=orange!80!black,line width=.22pt,fill opacity=.28},
strip/.style={fill=gray!8,draw=gray!45,line width=.18pt},
marker/.style={gray!45,densely dashed,line width=.18pt},
crossing/.style={fill=teal!60,draw=none,fill opacity=.30,rounded corners=.5pt},
crossborder/.style={draw=teal!80!black,line width=.62pt,rounded corners=.5pt},
copylabel/.style={font=\scriptsize}
]
\def\s{0.4}
\def\w{0.68}

\draw[gray!55,line width=.25pt] (0,0) rectangle (12,12);

\foreach \i/\xi in {0/3,1/11,2/19,3/27} {
  \pgfmathsetmacro{\c}{\s*\xi}
  \pgfmathsetmacro{\a}{\c-\w/2}
  \pgfmathsetmacro{\b}{\c+\w/2}
  \draw[strip] (0,\a) rectangle (12,\b);
  \node[copylabel,anchor=east] at (-.18,\c) {$H_{\i}$};
}

\foreach \j/\xi in {0/3,1/11,2/19,3/27} {
  \pgfmathsetmacro{\c}{\s*\xi}
  \pgfmathsetmacro{\a}{\c-\w/2}
  \pgfmathsetmacro{\b}{\c+\w/2}
  \draw[strip] (\a,0) rectangle (\b,12);
  \node[copylabel,anchor=south] at (\c,12.18) {$V_{\j}$};
}

\foreach \i/\xi in {0/3,1/11,2/19,3/27} {
  \pgfmathsetmacro{\row}{\s*\xi-\w/2}
  \foreach \xa/\xb/\ya/\yb in {
    0/30/26/28,
    0/16/20/22,
    14/30/20/22,
    0/8/2/10,
    6/16/8/14,
    14/24/8/14,
    22/30/2/12,
    6/24/0/4
  } {
    \pgfmathsetmacro{\xone}{\s*\xa}
    \pgfmathsetmacro{\xtwo}{\s*\xb}
    \pgfmathsetmacro{\yone}{\row+\w*\ya/30}
    \pgfmathsetmacro{\ytwo}{\row+\w*\yb/30}
    \filldraw[hrect] (\xone,\yone) rectangle (\xtwo,\ytwo);
  }
}

\foreach \j/\xi in {0/3,1/11,2/19,3/27} {
  \pgfmathsetmacro{\col}{\s*\xi-\w/2}
  \foreach \xa/\xb/\ya/\yb in {
    0/30/26/28,
    0/16/20/22,
    14/30/20/22,
    0/8/2/10,
    6/16/8/14,
    14/24/8/14,
    22/30/2/12,
    6/24/0/4
  } {
    \pgfmathsetmacro{\xone}{\col+\w*\ya/30}
    \pgfmathsetmacro{\xtwo}{\col+\w*\yb/30}
    \pgfmathsetmacro{\yone}{\s*\xa}
    \pgfmathsetmacro{\ytwo}{\s*\xb}
    \filldraw[vrect] (\xone,\yone) rectangle (\xtwo,\ytwo);
  }
}

\foreach \i/\xih in {0/3,1/11,2/19,3/27} {
  \foreach \j/\xiv in {0/3,1/11,2/19,3/27} {
    \pgfmathsetmacro{\yc}{\s*\xih}
    \pgfmathsetmacro{\xc}{\s*\xiv}
    \fill[crossing] (\xc-\w*.62,\yc-\w*.62) rectangle (\xc+\w*.62,\yc+\w*.62);
  }
}

\foreach \i/\xi in {0/3,1/11,2/19,3/27} {
  \pgfmathsetmacro{\c}{\s*\xi}
  \draw[marker] (0,\c) -- (12,\c);
}
\foreach \j/\xi in {0/3,1/11,2/19,3/27} {
  \pgfmathsetmacro{\c}{\s*\xi}
  \draw[marker] (\c,0) -- (\c,12);
}

\foreach \i/\xih in {0/3,1/11,2/19,3/27} {
  \foreach \j/\xiv in {0/3,1/11,2/19,3/27} {
    \pgfmathsetmacro{\yc}{\s*\xih}
    \pgfmathsetmacro{\xc}{\s*\xiv}
    \draw[crossborder] (\xc-\w*.62,\yc-\w*.62) rectangle (\xc+\w*.62,\yc+\w*.62);
  }
}
\end{tikzpicture}
\caption{The sixteen crossings in the first crossbar. Each highlighted square
is where one horizontal copy strip and one vertical copy strip meet.}
\label{fig:app-crossings-geometry}
\end{figure}

\clearpage

\Cref{fig:app-crossing-zooms} enlarges one crossing of each type. The type is
determined only by the two slot-class sizes $|\mathcal X_j|$ and
$|\mathcal X_i|$; these are the two numbers shown in the board and in the
complete-join schematic below.

\begin{figure}[H]
\centering
\begin{tikzpicture}[
type33/.style={draw=crossTypeAA,fill=crossTypeAAFill,line width=.4pt},
type34/.style={draw=crossTypeAB,fill=crossTypeABFill,line width=.4pt},
type43/.style={draw=crossTypeBA,fill=crossTypeBAFill,line width=.4pt},
type44/.style={draw=crossTypeBB,fill=crossTypeBBFill,line width=.4pt},
panel/.style={rounded corners=4pt},
hrect/.style={fill=blue!35,draw=blue!75!black,line width=.18pt,fill opacity=.28},
vrect/.style={fill=orange!42,draw=orange!80!black,line width=.18pt,fill opacity=.28},
strip/.style={fill=gray!8,draw=gray!45,line width=.14pt},
marker/.style={gray!45,densely dashed,line width=.14pt},
crossing/.style={fill=teal!60,draw=none,fill opacity=.30,rounded corners=.5pt},
crossborder/.style={draw=teal!80!black,line width=.5pt,rounded corners=.5pt},
lab/.style={font=\scriptsize},
tinylabel/.style={font=\tiny}
]
\def\s{0.4}
\def\w{0.68}
\def\r{0.82}
\def\zoom{1.72}

\newcommand{\drawsource}{%
  \draw[gray!55,line width=.18pt] (0,0) rectangle (12,12);
  \foreach \i/\xi in {0/3,1/11,2/19,3/27} {
    \pgfmathsetmacro{\c}{\s*\xi}
    \pgfmathsetmacro{\a}{\c-\w/2}
    \pgfmathsetmacro{\b}{\c+\w/2}
    \draw[strip] (0,\a) rectangle (12,\b);
  }
  \foreach \j/\xi in {0/3,1/11,2/19,3/27} {
    \pgfmathsetmacro{\c}{\s*\xi}
    \pgfmathsetmacro{\a}{\c-\w/2}
    \pgfmathsetmacro{\b}{\c+\w/2}
    \draw[strip] (\a,0) rectangle (\b,12);
  }
  \foreach \i/\xi in {0/3,1/11,2/19,3/27} {
    \pgfmathsetmacro{\row}{\s*\xi-\w/2}
    \foreach \xa/\xb/\ya/\yb in {
      0/30/26/28,
      0/16/20/22,
      14/30/20/22,
      0/8/2/10,
      6/16/8/14,
      14/24/8/14,
      22/30/2/12,
      6/24/0/4
    } {
      \pgfmathsetmacro{\xone}{\s*\xa}
      \pgfmathsetmacro{\xtwo}{\s*\xb}
      \pgfmathsetmacro{\yone}{\row+\w*\ya/30}
      \pgfmathsetmacro{\ytwo}{\row+\w*\yb/30}
      \filldraw[hrect] (\xone,\yone) rectangle (\xtwo,\ytwo);
    }
  }
  \foreach \j/\xi in {0/3,1/11,2/19,3/27} {
    \pgfmathsetmacro{\col}{\s*\xi-\w/2}
    \foreach \xa/\xb/\ya/\yb in {
      0/30/26/28,
      0/16/20/22,
      14/30/20/22,
      0/8/2/10,
      6/16/8/14,
      14/24/8/14,
      22/30/2/12,
      6/24/0/4
    } {
      \pgfmathsetmacro{\xone}{\col+\w*\ya/30}
      \pgfmathsetmacro{\xtwo}{\col+\w*\yb/30}
      \pgfmathsetmacro{\yone}{\s*\xa}
      \pgfmathsetmacro{\ytwo}{\s*\xb}
      \filldraw[vrect] (\xone,\yone) rectangle (\xtwo,\ytwo);
    }
  }
  \foreach \i/\xi in {0/3,1/11,2/19,3/27} {
    \pgfmathsetmacro{\c}{\s*\xi}
    \draw[marker] (0,\c) -- (12,\c);
  }
  \foreach \j/\xi in {0/3,1/11,2/19,3/27} {
    \pgfmathsetmacro{\c}{\s*\xi}
    \draw[marker] (\c,0) -- (\c,12);
  }
}

\newcommand{\zoompanel}[8]{%
  \begin{scope}[xshift=#1cm,yshift=#2cm]
  \draw[panel,#3] (-.16,-.34) rectangle (3.02,3.38);
  \node[lab] at (1.43,3.12) {$#4$};
  \node[tinylabel] at (1.43,-.12) {$#8$};
  \begin{scope}[xshift=.22cm,yshift=.28cm]
    \fill[white] (0,0) rectangle (2.42,2.42);
    \draw[black!35,line width=.25pt] (0,0) rectangle (2.42,2.42);
    \clip (0,0) rectangle (2.42,2.42);
    \pgfmathsetmacro{\xs}{-\zoom*(#5-\r)}
    \pgfmathsetmacro{\ys}{-\zoom*(#6-\r)}
    \begin{scope}[xshift=\xs cm,yshift=\ys cm,scale=\zoom]
      \pgfmathsetmacro{\xl}{#5-\w*.62}
      \pgfmathsetmacro{\xr}{#5+\w*.62}
      \pgfmathsetmacro{\yb}{#6-\w*.62}
      \pgfmathsetmacro{\yt}{#6+\w*.62}
      \drawsource
      \fill[crossing] (\xl,\yb) rectangle (\xr,\yt);
      \foreach \i/\xi in {0/3,1/11,2/19,3/27} {
        \pgfmathsetmacro{\c}{\s*\xi}
        \draw[marker] (0,\c) -- (12,\c);
      }
      \foreach \j/\xi in {0/3,1/11,2/19,3/27} {
        \pgfmathsetmacro{\c}{\s*\xi}
        \draw[marker] (\c,0) -- (\c,12);
      }
      \draw[crossborder] (\xl,\yb) rectangle (\xr,\yt);
    \end{scope}
  \end{scope}
  \draw[black!35,line width=.25pt] (.22,.28) rectangle (2.64,2.7);
  \end{scope}
}

\zoompanel{.1}{3.85}{type33}{K_{3,3}}{1.2}{1.2}{3}{H_0\text{ with }V_0}
\zoompanel{4.0}{3.85}{type34}{K_{3,4}}{1.2}{4.4}{4}{H_1\text{ with }V_0}
\zoompanel{.1}{0}{type43}{K_{4,3}}{4.4}{1.2}{3}{H_0\text{ with }V_1}
\zoompanel{4.0}{0}{type44}{K_{4,4}}{4.4}{4.4}{4}{H_1\text{ with }V_1}
\end{tikzpicture}
\caption{Four enlarged crossings, one of each size type. The colors match the
board colors used in \Cref{fig:app-mixed-cell-graph}.}
\label{fig:app-crossing-zooms}
\end{figure}
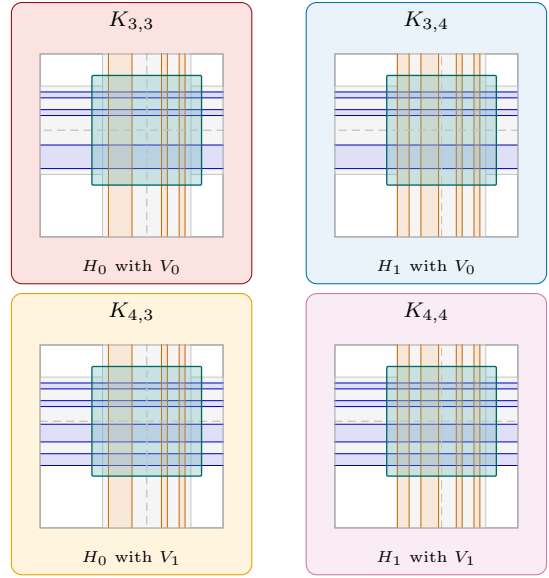

We now pass from the geometry of the highlighted crossings to the graph
structure inside them. For a slot $a\in[4]$, write
\[ \mathcal X_a=\{R\in P_0:a\in X(R)\}. \]
In the part of the drawing where $H_i(P_0)$ crosses $V_j(P_0)$, the crossing
rule says that the new edges are exactly the complete bipartite graph between
$H_i(\mathcal X_j)$ and $V_j(\mathcal X_i)$. Since
$|\mathcal X_0|=|\mathcal X_3|=3$ and
$|\mathcal X_1|=|\mathcal X_2|=4$, only four sizes can occur. Here
\[
\mathcal X_0=\{C_0,U_L,T\},\quad
\mathcal X_1=\{C_1,Q,U_L,T\},\quad
\mathcal X_2=\{C_2,Q,U_R,T\},\quad
\mathcal X_3=\{C_3,U_R,T\}.
\]
The figure below colors the $4$ by $4$ board by these four types and then
draws one representative schematic for each type, recording each complete
bipartite graph as a shaded block instead of drawing all its edges. A labeled
copy of $F$ is included as a key for the small unlabeled copies. The yellow
vertices in the representative panels are the slot classes that form the
complete join. We also mark the independent set
$I_0=\{T,U_L,C_0,C_2\}$ in red. In the representative panels the red mark is
put only on the horizontal side.

\begin{figure}[H]
\centering
\begin{tikzpicture}[
scale=.9,
type33/.style={draw=crossTypeAA,fill=crossTypeAAFill,line width=.35pt},
type34/.style={draw=crossTypeAB,fill=crossTypeABFill,line width=.35pt},
type43/.style={draw=crossTypeBA,fill=crossTypeBAFill,line width=.35pt},
type44/.style={draw=crossTypeBB,fill=crossTypeBBFill,line width=.35pt},
panel/.style={rounded corners=4pt,line width=.45pt},
fnode/.style={circle,draw,fill=white,inner sep=1.45pt},
fhit/.style={circle,draw,fill=yellow!45,inner sep=1.45pt},
keynode/.style={circle,draw,fill=white,inner sep=.85pt,font=\tiny},
keyind/.style={circle,draw=red!85!black,fill=red!62,line width=.5pt,inner sep=.85pt,font=\tiny},
iwit/.style={circle,fill=red!75!black,draw=white,line width=.12pt,inner sep=.55pt},
fedge/.style={black!70,line width=.3pt},
joinblock/.style={draw=black!55,fill=white,opacity=.62,rounded corners=3pt,line width=.32pt},
lab/.style={font=\scriptsize},
tinylabel/.style={font=\tiny}
]

\newcommand{\joinpanel}[8]{%
  \begin{scope}[xshift=#1cm,yshift=#2cm]
  \draw[panel,#3] (-1.72,-.88) rectangle (1.72,.9);
  \node[lab] at (0,.67) {$#4$};

  \coordinate (LCzero) at (-1.42,-.24);
  \coordinate (LCone) at (-1.16,.02);
  \coordinate (LCtwo) at (-.88,.02);
  \coordinate (LCthree) at (-.62,-.24);
  \coordinate (LQ) at (-1.02,-.5);
  \coordinate (LUL) at (-1.22,.34);
  \coordinate (LUR) at (-.82,.34);
  \coordinate (LT) at (-1.02,.58);

  \coordinate (RCzero) at (.62,-.24);
  \coordinate (RCone) at (.88,.02);
  \coordinate (RCtwo) at (1.16,.02);
  \coordinate (RCthree) at (1.42,-.24);
  \coordinate (RQ) at (1.02,-.5);
  \coordinate (RUL) at (.82,.34);
  \coordinate (RUR) at (1.22,.34);
  \coordinate (RT) at (1.02,.58);

  \draw[joinblock] (-.28,-.55) rectangle (.28,.42);
  \node[tinylabel,align=center] at (0,-.08) {complete\\join};

  \draw[fedge] (LCzero) -- (LCone) -- (LCtwo) -- (LCthree) -- (LQ) -- (LCzero);
  \draw[fedge] (LUL) -- (LUR);
  \draw[fedge] (RCzero) -- (RCone) -- (RCtwo) -- (RCthree) -- (RQ) -- (RCzero);
  \draw[fedge] (RUL) -- (RUR);

  \foreach \name in {Czero,Cone,Ctwo,Cthree,Q,UL,UR,T} {
    \node[fnode] at (L\name) {};
    \node[fnode] at (R\name) {};
  }
  \foreach \name in {#7} {
    \node[fhit] at (L\name) {};
  }
  \foreach \name in {#8} {
    \node[fhit] at (R\name) {};
  }
  \foreach \name in {Czero,Ctwo,UL,T} {
    \node[iwit] at (L\name) {};
  }
  \node[tinylabel] at (-1.02,-.73) {$|\mathcal X_j|=#5$};
  \node[tinylabel] at (1.02,-.73) {$|\mathcal X_i|=#6$};
  \end{scope}
}

\joinpanel{2.25}{.1}{type33}{K_{3,3}}{3}{3}{Czero,UL,T}{Czero,UL,T}
\joinpanel{6.25}{.1}{type34}{K_{3,4}}{3}{4}{Czero,UL,T}{Cone,Q,UL,T}
\joinpanel{2.25}{-2.05}{type43}{K_{4,3}}{4}{3}{Cone,Q,UL,T}{Czero,UL,T}
\joinpanel{6.25}{-2.05}{type44}{K_{4,4}}{4}{4}{Cone,Q,UL,T}{Cone,Q,UL,T}

\begin{scope}[xshift=.8cm,yshift=2.4cm]
\node[lab] at (1.23,3.7) {crossings};
\foreach \j/\x in {0/0,1/.82,2/1.64,3/2.46} {
  \node[lab] at (\x,3.3) {$V_{\j}$};
}
\foreach \i/\y/\ai in {3/2.64/3,2/1.92/4,1/1.2/4,0/.48/3} {
  \node[lab,anchor=east] at (-.58,\y) {$H_{\i}$};
}
\node[type33,minimum width=.76cm,minimum height=.55cm,font=\tiny] at (0,2.64) {$3\times3$};
\node[type43,minimum width=.76cm,minimum height=.55cm,font=\tiny] at (.82,2.64) {$4\times3$};
\node[type43,minimum width=.76cm,minimum height=.55cm,font=\tiny] at (1.64,2.64) {$4\times3$};
\node[type33,minimum width=.76cm,minimum height=.55cm,font=\tiny] at (2.46,2.64) {$3\times3$};
\node[type34,minimum width=.76cm,minimum height=.55cm,font=\tiny] at (0,1.92) {$3\times4$};
\node[type44,minimum width=.76cm,minimum height=.55cm,font=\tiny] at (.82,1.92) {$4\times4$};
\node[type44,minimum width=.76cm,minimum height=.55cm,font=\tiny] at (1.64,1.92) {$4\times4$};
\node[type34,minimum width=.76cm,minimum height=.55cm,font=\tiny] at (2.46,1.92) {$3\times4$};
\node[type34,minimum width=.76cm,minimum height=.55cm,font=\tiny] at (0,1.2) {$3\times4$};
\node[type44,minimum width=.76cm,minimum height=.55cm,font=\tiny] at (.82,1.2) {$4\times4$};
\node[type44,minimum width=.76cm,minimum height=.55cm,font=\tiny] at (1.64,1.2) {$4\times4$};
\node[type34,minimum width=.76cm,minimum height=.55cm,font=\tiny] at (2.46,1.2) {$3\times4$};
\node[type33,minimum width=.76cm,minimum height=.55cm,font=\tiny] at (0,.48) {$3\times3$};
\node[type43,minimum width=.76cm,minimum height=.55cm,font=\tiny] at (.82,.48) {$4\times3$};
\node[type43,minimum width=.76cm,minimum height=.55cm,font=\tiny] at (1.64,.48) {$4\times3$};
\node[type33,minimum width=.76cm,minimum height=.55cm,font=\tiny] at (2.46,.48) {$3\times3$};
\node[lab,align=center] at (1.23,-.12)
{$H_i(\mathcal X_j)$ joined to $V_j(\mathcal X_i)$};
\end{scope}

\begin{scope}[xshift=6.35cm,yshift=2.75cm]
\draw[gray!60,rounded corners=3pt,line width=.35pt] (-.45,-.9) rectangle (2.5,2.15);
\node[lab] at (1.02,2.42) {base graph $F$};
\coordinate (KCzero) at (0,0);
\coordinate (KCone) at (.65,.5);
\coordinate (KCtwo) at (1.4,.5);
\coordinate (KCthree) at (2.05,0);
\coordinate (KQ) at (1.02,-.58);
\coordinate (KUL) at (.6,1.13);
\coordinate (KUR) at (1.45,1.13);
\coordinate (KT) at (1.02,1.72);
\draw[fedge] (KCzero) -- (KCone) -- (KCtwo) -- (KCthree) -- (KQ) -- (KCzero);
\draw[fedge] (KUL) -- (KUR);
\node[keyind] at (KCzero) {$C_0$};
\node[keynode] at (KCone) {$C_1$};
\node[keyind] at (KCtwo) {$C_2$};
\node[keynode] at (KCthree) {$C_3$};
\node[keynode] at (KQ) {$Q$};
\node[keyind] at (KUL) {$U_L$};
\node[keynode] at (KUR) {$U_R$};
\node[keyind] at (KT) {$T$};
\end{scope}
\end{tikzpicture}
\caption{A schematic view of the complete bipartite pieces between horizontal
and vertical copies in $P_1$. Yellow marks the slot classes that are joined;
red marks the base independent set on the horizontal side.}
\label{fig:app-mixed-cell-graph}
\end{figure}
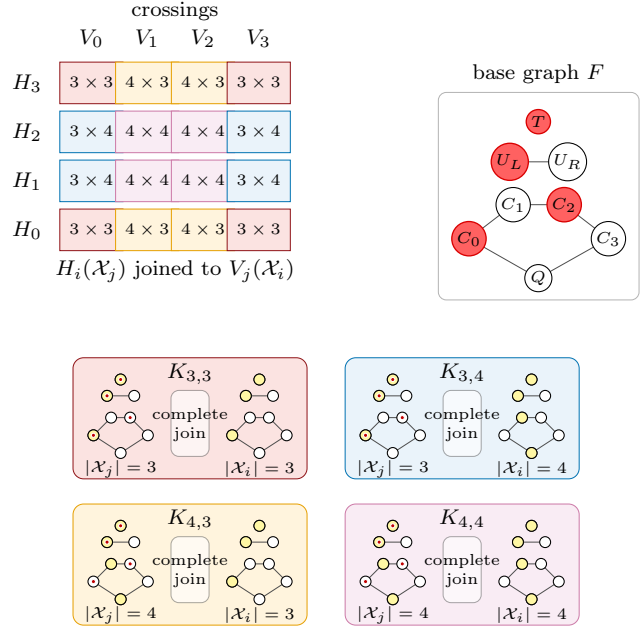

The red vertices give the matching lower-bound witness for
$\alpha(P_1)=16$. Take these four vertices in each of the four horizontal
copies and take no vertices from the vertical copies. The horizontal copies
are pairwise disjoint, and the chosen vertices are independent inside each
copy of $F$, so this gives an independent set of size $4\cdot4=16$. The
opposite inequality is the board-counting argument proved in
\Cref{prop:independence-recursion}; the schematic here is meant to make the
witness and the complete joins at the crossings visible.

\section{One-copy profile checks}\label{app:profile-checks}

This appendix records the finite checks used in \Cref{sec:finite-lp-gaps}. The
checks all take place inside $P_1$. Since $P_1$ is triangle-free, every clique
is either a single rectangle or an edge. Thus, for each profile, feasibility and
the slot loads can be checked by inspecting vertices and edges.

For a profile $w$ on $P_1$, let $X_w(b)$ denote the maximum of $X_w(a,b)$ over
the first digit $a$. The profiles used for the first $P_2$ gap have the
following values and final-digit $x$-loads:
\[
\begin{array}{c|c|c|c}
\text{profile}&w(P_1)&\max_K w(K)&(X_w(0),X_w(1),X_w(2),X_w(3))\\ \hline
O&0&0&(0,0,0,0)\\
A&14&1/2&(1/2,1/2,3/8,1/4)\\
B&16&3/4&(1/2,1/2,1/2,1/2)\\
C&26&1&(1,1,1/2,1/2)\\
D&27&1&(1,1,5/8,1/2)\\
E&28&1&(1,1,3/4,1/2).
\end{array}
\]
Here $\max_K w(K)$ is the maximum over cliques $K$ of $P_1$; in particular,
each listed profile is feasible.

The reusable vector $A_2$ uses the following four profiles:
\[
\begin{array}{c|c|c|c}
\text{profile}&w(P_1)&\max_K w(K)&(X_w(0),X_w(1),X_w(2),X_w(3))\\ \hline
R&13&1/2&(1/2,1/2,1/4,1/4)\\
S&16&1&(1/2,1/2,1/2,1/2)\\
U&10&1&(1,1,0,0)\\
V&12&1&(1,1,1/2,0).
\end{array}
\]

For $U$ and $V$ we also use $y$-loads and joint loads. For this table only,
write
\[
\widehat Z_w(b)=\max\{Z_w((a,b),t):a\in[4],\ t\in[4]^2\}.
\]
Then
\[
\begin{array}{c|c|c}
\text{profile}&\max_t Y_w(t)&(\widehat Z_w(0),\widehat Z_w(1),\widehat Z_w(2),\widehat Z_w(3))\\ \hline
U&1/2&(1/2,1/2,0,0)\\
V&1/2&(1/2,1/2,1/4,0).
\end{array}
\]
These are exactly the bounds used in \eqref{eq:p3-extra-loads}.

\end{document}